\documentclass[12pt,oneside]{amsart}
\usepackage{a4wide,enumerate,xcolor}
\usepackage{amsmath,graphicx,comment,enumerate}
\usepackage{tikz,geometry} 
\usepackage{amssymb} 
\allowdisplaybreaks

\let\pa\partial
\let\na\nabla
\let\eps\varepsilon
\newcommand{\N}{{\mathbb N}}
\newcommand{\R}{{\mathbb R}}
\newcommand{\diver}{\operatorname{div}}

\newcommand{\E}{\mathcal{E}}
\newcommand{\F}{\mathcal{F}}
\newcommand{\T}{\mathcal{T}}
\newcommand{\m}{\mathrm{m}}
\newcommand{\D}{\mathrm{D}}
\newcommand{\dist}{\mathrm{d}}

\newtheorem{theorem}{Theorem}
\newtheorem{lemma}[theorem]{Lemma}
\newtheorem{proposition}[theorem]{Proposition}

\newtheorem{definition}{Definition}

\begin{document}

\title[Scharfetter--Gummel scheme for drift--diffusion models]{
A convergent Scharfetter--Gummel scheme \\
for a three-Species drift--diffusion model \\
for memristors}

\author[A. J\"ungel]{Ansgar J\"ungel}
\address{Institute of Analysis and Scientific Computing, TU Wien, Wiedner Hauptstra\ss e 8--10, 1040 Wien, Austria}
\email{juengel@tuwien.ac.at} 

\author[Z. Sun]{Zhiwei Sun}
\address{Institute of Analysis and Scientific Computing, TU Wien, Wiedner Hauptstra\ss e 8--10, 1040 Wien, Austria}
\email{zhiwei.sun@tuwien.ac.at} 

\author[S. Xhahysa]{Sara Xhahysa}
\address{Institute of Analysis and Scientific Computing, TU Wien, Wiedner Hauptstra\ss e 8--10, 1040 Wien, Austria}
\email{sara.xhahysa@tuwien.ac.at} 

\date{\today}

\thanks{The authors acknowledge partial support from the Austrian Science Fund (FWF), grant 10.55776/PAT2687825, and from the Austrian Federal Ministry for Women, Science and Research and implemented by \"OAD, grant MULT09/2025. This work has received funding from the European Research Council (ERC) under the European Union's Horizon 2020 research and innovation programme, ERC Advanced Grant NEUROMORPH, no.~101018153. For open-access purposes, the authors have applied a CC BY public copyright license to any author-accepted manuscript version arising from this submission.} 

\begin{abstract}\sloppy
A structure-preserving fully implicit Scharfetter--Gummel finite-volume scheme for a three-species drift--diffusion model for semiconductors is proposed and analyzed. The equations describe the evolution of the electron, hole, and oxygen vacancy densities in a (bounded) memristor device, coupled to the Poisson equation for the electric potential, with mixed-type boundary conditions. Recasting the Scharfetter--Gummel fluxes in an upwind form, a hidden Fisher information component is revealed. Owing to the degeneracy of the Bernoulli function appearing in the fluxes, additional edgewise coercivity estimates are required, leading to refined local and global dissipation estimates. Using these ideas, the existence of a discrete finite-volume solution, a discrete free energy inequality, and the convergence of the numerical scheme are established. Numerical simulations in two space dimensions confirm the structure-preserving properties of the scheme and illustrate the filament formation in a memristor device.
\end{abstract}

 

\maketitle


\section{Introduction}

Memristor devices have emerged as promising candidates for next-generation nonvolatile memory and neuromorphic computing architectures due to their nanoscale size, low power consumption, and intrinsic resistive switching behavior \cite{XJZKJWY23}. Drift-diffusion models provide a physically grounded framework for capturing the interplay between carrier transport, electrostatic effects, and resistive state evolution \cite{GSTD13}. The convection-dominated character of the governing equations often leads to severe numerical instabilities when standard discretization techniques are employed. To address these challenges, we employ and analyze a finite-volume discretization based on the Scharfetter--Gummel scheme \cite{ScGu69,SBW09}. The scheme preserves mass under no-flux boundary conditions, positivity, and free-energy dissipation, even in drift-dominated regimes, and is shown to be stable and convergent. The convergence proof combines a refined dissipation structure of the Scharfetter--Gummel flux with two complementary coercivity estimates, yielding the local and global Fisher-information bounds required for compactness.

\subsection{Model equations}

The scaled equations for the electron density $N$, hole density $P$, oxygen vacancy density $Q$, and electric potential $V$ are given by
\begin{equation}\label{1.eq}
\left\{\begin{aligned}
  &\pa_t N - \diver J_N = 0, \quad J_N = \na N-N\na V,  \\
  &\pa_t P - \diver J_P = 0, \quad J_P = \na P+P\na V,  \\
  &\pa_t Q - \diver J_Q = 0, \quad J_Q = \na Q+Q\na V,  \\
  &\lambda^2\Delta V = N-P-Q + A(x)\quad\mbox{in }\Omega,\ t>0,
\end{aligned}\right.
\end{equation}
where $\Omega\subset\R^d$ ($d\ge 1$) is a bounded domain with boundary $\pa\Omega=\Gamma_D\cup\Gamma_N$, supplemented with initial and mixed boundary conditions:
\begin{equation}\label{1.bic}
\left\{\begin{aligned}
  N(0,\cdot)=N^{\rm in}, \quad P(0,\cdot)=P^{\rm in}, \quad
  Q(0,\cdot)=Q^{\rm in} &\quad\mbox{in }\Omega, \\
  N=N^D, \quad P=P^D, \quad V=V^D &\quad\mbox{on }\Gamma_D,\ t>0, \\
  J_N\cdot\nu = J_P\cdot\nu = \na V\cdot\nu = 0 &\quad\mbox{on }
  \Gamma_N,\ t>0, \\
  J_Q\cdot\nu = 0 &\quad\mbox{on }\pa\Omega,\ t>0.
\end{aligned}\right.
\end{equation}
The parameter $\lambda>0$ denotes the scaled Debye length, $A(x)$ is the given immobile acceptor density, and $J_N$, $J_P$, $J_Q$ are the fluxes of the electrons, holes, oxygen vacancies, respectively. We prescribe the electron and holes densities, as well as the electric voltage at the Ohmic contacts $\Gamma_D$, while $\Gamma_N$ represents the union of insulating boundary parts. The Dirichlet conditions for the carrier densities are derived under the assumptions of charge neutrality and thermal equilibrium \cite[Sec.~5.1]{Sel84}. The boundary is assumed to be impermeable to oxygen vacancies; accordingly, we impose no-flux boundary conditions on $Q$. 

\subsection{Main challenges}

A distinctive feature of \eqref{1.eq}--\eqref{1.bic} is that the third species $Q$ satisfies a no-flux condition on the whole boundary, whereas
$N$ and $P$ solve mixed boundary conditions. This mixed-type boundary structure is one of the key sources of analytical and numerical difficulties in the present problem, destroying the monotonicity structure of the drift term (see, e.g., \cite[(1.9)]{JJZ23}). For the existence analysis for \eqref{1.eq}--\eqref{1.bic}, these issues were overcome in \cite{JJZ23} by exploiting the free energy structure and by combining the local strong convergence of the densities and global weak convergence arguments. More precisely, the free energy (or entropy)
\begin{align*}
  E &= \int_\Omega\bigg\{N\bigg(\log\frac{N}{N^D}-1\bigg)
  + P\bigg(\log\frac{P}{P^D}-1\bigg)
  + Q(\log Q+V^D-1) \\
  &\phantom{xx}+ \frac{\lambda^2}{2} |\na(V-V^D)|^2\bigg\}dx,
\end{align*}
consisting of the internal energies of the species and the electric energy, satisfies the free energy (or entropy) inequality
\begin{align}\label{1.ei}
  \frac{dE}{dt} &+ \frac12\int_\Omega\bigg(|2\na\sqrt{N}-\sqrt{N}\na V|^2
  + |2\na\sqrt{P}+\sqrt{P}\na V|^2 \\
  &+ |2\na\sqrt{Q}+\sqrt{Q}\na V|^2
  \bigg)dx \le C, \nonumber 
\end{align}
where $C>0$ depends on the boundary data. The expression $|2\na\sqrt{U}\pm\sqrt{U}\na V|^2 = U|\na(\log U\pm V)|^2$ with $U=N,P,Q$ is called dissipation. The energy structure yields bounds for the Fisher information $|\na\sqrt{U}|^2$, which is a key bound for the existence analysis. 

Let $\mathcal{T}$ be a triangulation of the domain $\Omega$ consisting of cells $K\in\mathcal{T}$,  $(N_K)_{K\in\T}$ and $(V_K)_{K\in\T}$ be cellwise constant approximations of $N$ and $V$, respectively, and let $\sigma=K|L$ be an edge or face between two cells $K$ and $L$. Then the classical Scharfetter--Gummel discretization of the electron flux $J_N$ is defined by
\begin{align*}
  \F_{K,\sigma} = \tau_\sigma\big(B(V_L-V_K)N_L - B(-(V_L-V_K))N_K\big),
\end{align*}
where $B(s)=s/(\exp(s)-1)$ for $s\neq 0$, $B(0)=1$ is the Bernoulli function and $\tau_\sigma$ is the transmissibility coefficient defined in \eqref{2.tau} below. This discretization balances drift and diffusion at the discrete level while preserving the entropy structure. The scheme was originally proposed for semiconductor equations \cite{ScGu69}, subsequently extended to multiple space dimensions \cite{FaGa91}, and later incorporated into the finite-volume framework \cite{Bes12,ChDr11}. More recently, it has been generalized to anisotropic convection--diffusion equations \cite{Que22}, and has been interpreted within the framework of generalized gradient flows \cite{HST24}. In the context of memristor models, a related finite-volume discretization was analyzed in \cite{abdel25}. The authors prove the existence of discrete solutions and entropy stability, but the convergence of the scheme has not been addressed.

The convergence analysis of Scharfetter--Gummel schemes is considerably more delicate than that of standard finite-volume discretizations in the presence of nonlinear drift terms. The main difficulty stems from the fact that the Bernoulli function depends exponentially on the discrete potential difference, rendering the flux control highly sensitive to the regularity of the drift potential. Consequently, early analytical results for Scharfetter--Gummel-type schemes were restricted to settings involving nonlinear diffusion coupled with linear drift, where the Bernoulli factor does not interact with a nonlinear potential \cite{Bes12}. 

When the drift is self-consistent and nonlinear, the analysis becomes more subtle. In \cite{CCFG21}, convergence arguments available for centred and Sedan fluxes could not be extended to the Scharfetter--Gummel flux in a degenerate drift--diffusion--Poisson model. The analysis in \cite{ScSe22} relies on Lipschitz regularity of the interaction potential, while \cite{CHM26} uses uniform bounds on the discrete densities and the resulting control of the Poisson potential. A different approach was developed in \cite{JLS26}, where an upwind representation of generalized Scharfetter--Gummel fluxes reveals entropy structures suitable for nonlocal cross-diffusion systems.

None of these mechanisms is available for the system considered here. The Poisson coupling leads to low-regularity solutions, while the asymmetric mixed-type boundary conditions significantly complicate the estimation of energy dissipation. In particular, the maximum-principle estimates and monotonicity arguments available for related two-species models with compatible boundary conditions cannot be used. 
Most importantly, although the standard Scharfetter--Gummel entropy production is nonnegative, it does not directly provide coercive estimates on the discrete square-root gradients needed to establish compactness. Overcoming this lack of coercivity is the main analytical challenge addressed in the present work.

\subsection{Main ideas and results}

In our setting, the only robust structure that remains available is the free energy inequality, showing that the entropy production $|2\na\sqrt{U}\pm \sqrt{U}\na V|^2$ ($U=N,P,Q$) and the electric energy density $|\na V|^2$ are uniformly bounded in $L^1(\Omega)$. In the continuous model, a uniform bound for $\na\sqrt{U}$ in $L^1(\Omega_T)$ was deduced in \cite{JJZ23} from entropy dissipation and the triangle inequality. Moreover, away from the boundary, a local bound for $\na\sqrt{U}$ in $L^2_{\rm loc}(\Omega)$ was derived.

Our numerical analysis follows the same principle. The difficulty is that, for the Scharfet\-ter--Gummel flux, the analog of the coercive square structure $|2\nabla \sqrt U \pm \sqrt U \nabla V|^2$ is not directly available. More precisely, the standard Scharfetter--Gummel entropy production is given by
\begin{align*}
  \mathcal P_{K,\sigma}[U,V] = \tau_\sigma\widetilde{U}_\sigma
  \big(\log U_L-\log U_K + V_L-V_K\big)^2,
\end{align*}
where the value $\widetilde{U}_\sigma$ on the edge $\sigma$ is defined in \eqref{2.flux2} below. While this expression is nonnegative, it is not coercive with respect to $\sqrt{U_L}-\sqrt{U_K}$. As a consequence, the classical dissipation structure does not yield the control of the discrete square-root gradient, or equivalently of the Fisher information, needed for the compactness argument.

The main novelty of the present work is the derivation of a refined dissipation structure for the Scharfetter--Gummel flux in the low-regularity Poisson-coupling framework. Rewriting the Scharfetter--Gummel flux in an upwind form (see Lemma \ref{lem.upwind}), we recover a square identity of the form
\begin{align*}
\mathcal P_{K,\sigma}[U,V]
= \tau_\sigma^*
\big[
B(|V_L-V_K|)(\sqrt{U_L}-\sqrt{U_K})
+ \cdots
\big]^2 ,
\end{align*}
where the omitted term contains the corresponding upwind drift contribution and where $\tau_\sigma^*$ is uniformly comparable to the transmissibility $\tau_\sigma$. This identity provides the missing link between the classical Scharfetter--Gummel free-energy dissipation and a discrete Fisher-information component, namely the square-root difference $\sqrt{U_L}-\sqrt{U_K}$.

However, this square identity alone is not sufficient. The Bernoulli factor $B(|V_L-V_K|)$ degenerates when the discrete potential difference becomes large, so that it does not directly yield a uniform estimate for the square-root gradient. To overcome this difficulty, we derive two complementary edgewise coercivity estimates. The first estimate is stronger, but contains a cross term involving the density and potential differences; it is combined with the discrete Poisson equation with a localized argument to obtain interior Fisher-information bounds. The second estimate is weaker but free of this cross term; it yields the global $L^1$-control of the square-root gradients needed for compactness and for recovering the Dirichlet traces. 
Both coercivity estimates are established at the level of the Scharfetter--Gummel flux for a general drift--diffusion equation.

Under natural assumptions on the initial data, boundary data, and doping profile, we prove the existence of positive discrete solutions (Theorem \ref{thm.ex}) and establish a discrete free-energy inequality (Proposition \ref{prop.dei}). We then show that, along a sub-sequence of admissible meshes with vanishing mesh size, the scheme converges to a weak solution of the three-species drift--diffusion system. More precisely, the reconstructed densities converge strongly in $L^1(\Omega_T)$, the discrete gradients of the square roots converge weakly in $L^1(\Omega_T;\R^d)$, and the potentials converge weakly* in $L^\infty(0,T;H^1(\Omega))$. The Dirichlet boundary conditions for the carrier densities are recovered at the level of the square-root variables (Theorem \ref{thm.conv}).

The paper is organized as follows. The finite-volume scheme and the main results are presented in Section \ref{sec.main}. Section \ref{sec.SG} reveals the refined dissipation structure of the Scharfetter--Gummel flux with the main coercivity estimates in Propositions \ref{prop.first} and \ref{prop.second}. The existence of a discrete solution as well as the global bounds from the discrete free energy inequality are proved in Section \ref{sec.ex}. We perform the localized coercivity analysis and derive the interior estimates for the Fisher information in Section \ref{sec.chi}. Section \ref{sec.comp} contains the compactness and convergence arguments to prove the convergence of the scheme. Some numerical experiments are presented in Section \ref{sec.num}. Finally, we prove two auxiliary results in Appendix \ref{sec.app}.


\section{Finite-volume scheme and main results}\label{sec.main}

We introduce the notation needed for the finite-volume method, detail our finite-volume scheme, and present the main results of the paper.

\subsection{Notations}

The domain $\Omega$ is discretized by an admissible triangulation in the sense of \cite[Definition 9.1]{EGH00}. The triangulation consists of a family $\T$ of open polygonal convex subsets of $\Omega$ (cells or control volumes), a family $\E$ of edges (or faces in three space dimensions), and a family of points $\mathcal{X}=(x_K)_{K\in\T}$ associated with the cells. The admissibility assumption implies that the line segment connecting the points $x_K$ and $x_L$ of two neighboring cells $K$ and $L$ is orthogonal to their common edge $\sigma=K|L$. 

The family of edges $\E$ is split into internal and external edges $\E=\E_{\rm int}\cup\E_{\rm ext}$, where $\E_{\rm int} = \{\sigma\in\E:\sigma\subset\Omega\}$ and $\E_{\rm ext} = \{\sigma\in\E:\sigma\subset\pa\Omega\}$. Each exterior edge is assumed to be an element of either the Dirichlet or Neumann boundary, setting $\E_{\rm ext} = \E_{\rm ext}^D\cup\E_{\rm ext}^N$. For given $K\in\T$, we define the set $\E_K$ of the edges of $K$, which consists of internal edges and possibly of edges on the Dirichlet or Neumann boundary, $\E_K= \E_{{\rm int},K}\cup\E_{{\rm ext},K}^D\cup\E_{{\rm ext},K}^N$. 

With the Euclidean distance $d(\cdot,\cdot)$, we introduce for any $\sigma\in\E$ the distance
\begin{align*}
  \dist_\sigma = \begin{cases}
  d(x_K,x_L) &\mbox{if }\sigma=K|L\in\E_{\rm int}, \\
  d(x_K,\sigma) &\mbox{if }\sigma\in\E_{\rm ext}.
  \end{cases}
\end{align*}
The transmissibility coefficient of the edge $\sigma$ is defined by
\begin{align}\label{2.tau}
  \tau_\sigma = \frac{\m(\sigma)}{\rm{d}_\sigma},
\end{align}
where $\m(\sigma)$ denotes the Lebesgue measure of $\sigma$. We assume that the meshes are uniformly regular in the sense that there exist $\xi_1$, $\xi_2>0$, independent of the mesh, such that
\begin{align}\label{2.meshreg}
  d(x_K,\sigma)\ge \xi_1\dist_\sigma, \quad
  \dist_\sigma\ge \xi_2\max_{\sigma'\in\E_K}\dist_{\sigma'}.
\end{align}
The first inequality is used to derive uniform refined-flux bounds (Lemma \ref{lem.flux}), and both inequalities are needed to estimate the cross terms (Lemma \ref{lem.second}) and the local square-root gradients (Lemma \ref{lem.grad}).

We use a uniform time discretization with time step $\Delta t>0$, and we set $t_k=k\Delta t$ for $k=1,\ldots,N_T$, where $T>0$, $N_T\in\N$, and $\Delta t=T/N_T$. We denote by $\mathcal{M}$ an admissible space--time discretization of $(0,T)\times\Omega$ composed of an admissible mesh $(\T,\E,\mathcal{X})$ and the values $(\Delta t,\Delta x)$, where $\Delta x = \max_{K\in\T}\mbox{diam}(K)$. The size of $\mathcal{M}$ is given by $\eta:=\max\{\Delta t,\Delta x\}$.

We define for $v\in\R^{\#\mathcal T}$ and $\sigma\in\E$ the difference operators
\begin{align*}
  \D_{K,\sigma}v = v_{K,\sigma}-v_K, \quad\mbox{where}\quad
  v_{K,\sigma} = \begin{cases}
  v_L &\mbox{if }\sigma=K|L\in\E_{{\rm int},K}, \\
  v_\sigma &\mbox{if }\sigma\in\E_{{\rm ext},K}.
  \end{cases}
\end{align*}
The value $v_\sigma$ equals $\m(\sigma)^{-1}\int_\sigma vds$ if $\sigma\in\E_{{\rm ext},K}^D$, while $v_\sigma$ is determined from the discrete no-flux condition $\F_{K,\sigma}=0$ if $\sigma\in\E_{{\rm ext},K}^N$, where $\F_{K,\sigma}$ is the numerical flux; see Section \ref{sec.flux} for details. In standard problems with mixed boundary conditions, the value $v_{K,\sigma}$ is defined as $v_K$ if $\sigma\in\E_{{\rm ext},K}^N$, but here the fact that different boundary conditions are imposed for $(N,P,V)$ and $Q$ requires a slightly different notation. 

We introduce the discrete $H^1(\Omega)$ seminorm and discrete $H^1(\Omega)$ norm as
\begin{align*}
  |v|_{1,2,\T} = \bigg(\sum_{\sigma\in\E}\tau_\sigma
  |\D_{K,\sigma}v|^2\bigg)^{1/2}, \quad
  \|v\|_{1,2,\T} = \big(\|v\|_{0,2,\T}^2 + |v|_{1,2,\T}^2\big)^{1/2},
\end{align*}
respectively, where the $L^p(\Omega)$ norm for $1\le p<\infty$ is given by
\begin{align*}
  \|v\|_{0,p,\T} = \bigg(\sum_{K\in\T}\m(K)|v_K|^p\bigg)^{1/p}.
\end{align*}
For the numerical flux $\F_{K,\sigma}$, the discrete integration-by-parts formula becomes 
\begin{align}\label{2.dibp}
  \sum_{K\in\T}\sum_{\sigma\in\E_K}\F_{K,\sigma}v_K
  = -\sum_{\sigma=K|L\in\E}\F_{K,\sigma}\D_{K,\sigma}v
  + \sum_{\sigma\in\E_{\rm ext}}\F_{K,\sigma}v_{\sigma}
  \quad\mbox{for }v\in \R^{\#\mathcal T}.
\end{align}

For the convergence result, we need a dual mesh $\T^*$ of $\Omega$. We associate with $K\in\T$ and $\sigma\in\E_K$ a dual cell $\Delta_\sigma\in\T^*$, defined by
\begin{itemize}
\item ``Diamond'': For $\sigma=K|L$, $\Delta_\sigma$ is the interior of the convex hull of $\sigma\cup\{x_K,x_L\}$. 
\item ``Triangle'': For $\sigma\in\E_{{\rm ext},K}$, the cell $\Delta_\sigma$ is the interior of the convex hull of $\sigma\cup\{x_K\}$. 
\end{itemize}
The cells $\Delta_\sigma$ define a partition of $\Omega$. It follows from the property that the line segment $\overline{x_Kx_L}$ between two neighboring centers of cells is orthogonal to the edge $\sigma=K|L$ that
\begin{align}\label{2.mdelta}
  d\m(\Delta_\sigma) = \m(\sigma)\dist_\sigma
  \quad\mbox{for }\sigma\in\E.
\end{align}
The approximate gradient of a function $v\in V_{\Delta x}$ is defined by
\begin{align}\label{2.grad}
  \na_\sigma^{\Delta x}v(x) = d\frac{\D_{K,\sigma}v}{\dist_\sigma}
  \nu_{K,\sigma}\quad\mbox{for }x\in\Delta_\sigma,
\end{align}
where $\nu_{K,\sigma}$ is the unit normal vector to $\sigma$ pointing outward of $K$. 

Finally, we introduce the following reconstruction operators. Let $u=(u_K)_{K\in\T}$ and $v=(v_\sigma)_{\sigma\in\E}$ be given. Then we define
\begin{align*}
  \pi_\eta u(t,x) = u_K 
  &\quad\mbox{for }(t,x)\in (t_{k-1},t_k]\times K, \\
  \pi_\eta^* v(t,x) = v_\sigma 
  &\quad\mbox{for }(t,x)\in(t_{k-1},t_k]\times\Delta_\sigma.
\end{align*}


\subsection{Fully implicit finite-volume scheme}

We define the discrete and boundary data by
\begin{align*}
  U^0_K := \frac{1}{\m(K)}\int_K U^{\rm in}(x)dx
  &\quad\mbox{for }U=N,P,Q, \\
  U_\sigma^D := \frac{1}{\m(\sigma)}\int_\sigma U^Dds, \quad
  U_K^D := \frac{1}{\m(K)}\int_K U^D(x) dx
  &\quad\mbox{for }U=N,P,V
\end{align*}
and $K\in\T$, $\sigma\in\E$, as well as the doping profile averages
\begin{align*}
  A_K := \frac{1}{\m(K)}\int_K A(x)dx \quad\mbox{for }K\in\T. 
\end{align*}
Given the cellwise values $N^{k-1}=(N_K^{k-1})_{K\in\T}$, $P^{k-1}=(P_K^{k-1})_{K\in\T}$, $Q^{k-1}=(Q_K^{k-1})_{K\in\T}\in \R^{\#\mathcal T}$, we seek $(N,P,Q,V)\in (\R^{\#\mathcal T})^4$, solving the discrete drift--diffusion equations
\begin{equation}\label{3.NPQ}
\left\{\begin{aligned}
  \frac{\m(K)}{\Delta t}(N_K^k-N_K^{k-1})
  - \sum_{\sigma\in\E_K}\F_{K,\sigma}[N^k,-V^k] &= 0, \\
  \frac{\m(K)}{\Delta t}(P_K^k-P_K^{k-1})
  - \sum_{\sigma\in\E_K}\F_{K,\sigma}[P^k,V^k] &= 0, \\
  \frac{\m(K)}{\Delta t}(Q_K^k-Q_K^{k-1})
  - \sum_{\sigma\in\E_K}\F_{K,\sigma}[Q^k,V^k] &= 0
\end{aligned}\right.
\end{equation}
for $K\in\T$ and $k=1,\ldots,N_T$. The Scharfetter--Gummel fluxes are defined by
\begin{equation}\label{3.F}
\left\{\begin{aligned}
  \F_{K,\sigma}[N^k,-V^k] &= \tau_\sigma\big(B(\D_{K,\sigma}V^k)
  N_{K,\sigma}^k - B(-\D_{K,\sigma}V^k)N_K^k\big), \\
  \F_{K,\sigma}[P^k,V^k] &= \tau_\sigma\big(B(-\D_{K,\sigma}V^k)
  P_{K,\sigma}^k - B(\D_{K,\sigma}V^k)P_K^k\big), \\
  \F_{K,\sigma}[Q^k,V^k] &= \tau_\sigma\big(B(-\D_{K,\sigma}V^k)
  Q_{K,\sigma}^k - B(\D_{K,\sigma}V^k)Q_K^k\big)
  \quad\mbox{for }\sigma\in\E_k,
\end{aligned}\right.
\end{equation}
recalling the definition $B(s)=s/(e^s-1)$ for $s\neq 0$ and $B(0)=1$ of the Bernoulli function. 
The discrete electric potential is determined by
\begin{align}\label{3.V}
  \lambda^2\sum_{\sigma\in\E_K}\tau_\sigma\D_{K,\sigma}V^k
  = \m(K)(N_K^k-P_K^k-Q_K^k+A_K)
\end{align}
for $K\in\T$, $k=1,\ldots,N_T$. The boundary edge values are given by the mixed-type boundary conditions
\begin{equation}\label{3.bc}
\left\{\begin{aligned}
  N_\sigma^k = N_\sigma^D, \quad P_\sigma^k = P_\sigma^D, \quad
  V_\sigma^k = V_\sigma &\quad\mbox{for }\sigma\in\E_{\rm ext}^D, \\
  \F_{K,\sigma}[N^k,-V^k] = \F_{K,\sigma}[P^k,V^k] = 0, \quad
  V_{K,\sigma}^k = V_K^k
  &\quad\mbox{for }\sigma\in\E_{{\rm ext},K}^N,\ K\in\T, \\
  \F_{K,\sigma}[Q^k,V^k] = 0 
  &\quad\mbox{for }\sigma\in\E_{{\rm ext},K},\ K\in\T.
\end{aligned}\right.
\end{equation}


\subsection{Main results}

We impose the following hypotheses:

\begin{itemize}
\item[\bf (H1)] Boundary: $\pa\Omega=\Gamma_D\cup\Gamma_N$ is Lipschitz, $\Gamma_D\cap\Gamma_N=\emptyset$, $\m(\Gamma_D)>0$, and $\Gamma_N$ is open in $\pa\Omega$. 
\item[\bf (H2)] Initial data: $N^{\rm in}$, $P^{\rm in}$, $Q^{\rm in}\in L\log L(\Omega)$ are nonnegative and not identically zero. 
\item[\bf (H3)] Boundary data: $N^D$, $P^D\ge 0$, $V^D\in\R$ on $\Gamma_D$ admit extensions to $\Omega\cup\Gamma_D$, still denoted by $N^D$, $P^D$, $V^D$, satisfying $N^D$, $P^D>0$ in $\Omega$ and $\log N^D-V^D$, $\log P^D+V^D\in W^{1,\infty}(\Omega)$.
\item[\bf (H4)] Doping profile: $A\in L^2(\Omega)$.
\end{itemize}
The condition that none of the initial functions are identically zero in 
Hypothesis (H2) guarantees that the discrete solutions $(N_K,P_K,Q_K)$ are strictly positive for all $K\in\T$ and $k\ge 1$ (Lemma \ref{lem.pos}). The strict positivity allows us to use the logarithmic components $\log N_K$, $\log P_K$, and $\log Q_K$ of the entropy variables directly in the proof of the discrete entropy inequality (Proposition \ref{prop.dei}) without any regularization as needed in \cite{JJZ23}.  

First, we state an existence result for the discrete nonlinear system.

\begin{theorem}[Existence and positivity]\label{thm.ex}
Let Hypotheses (H1)--(H4) hold. For $k\in\{1,\ldots,$ $N_T\}$, system \eqref{3.NPQ}--\eqref{3.bc} admits a solution $(N^k,P^k,Q^k,V^k)\in (\R^{\#\mathcal T})^4$ satisfying
\begin{align*}
  N_K^k>0,\ P_K^k>0,\ Q_K^k>0 \quad\mbox{for }K\in\T,\ k=1,\ldots,N_T.
\end{align*}
\end{theorem}

For two cellwise positive functions $u=(u_K)_{K\in\T}$ and $v=(v_K)_{K\in\T}$, the discrete relative entropy is defined by
\begin{align}\label{3.H}
  \mathcal H(u|v) := \sum_{K\in\T}\m(K)\bigg(u_K\log\frac{u_K}{v_K}
  - u_K + v_K\bigg)\ge 0.
\end{align}
The discrete electric energy is given by
\begin{align}\label{3.G}
  \mathcal G(V^k;V^D) := \frac{\lambda^2}{2}\sum_{K\in\T}\sum_{\sigma\in\E_K}
  \tau_\sigma|\D_{K,\sigma}(V^k-V^D)|^2,
\end{align}
and the discrete free energy equals
\begin{align}\label{3.E}
  \mathcal E_{\rm free}^k := \mathcal H(N^k|N^D) + \mathcal H(P^k|P^D)
  + \mathcal H(Q^k|Q^D) + \mathcal G(V^k;V^D),
\end{align}
where $Q^D=(Q_K^D)_{K\in\T}$ with $Q_K^D:=\exp(-V_K^D)$ for $K\in\T$. Furthermore, with a potential $\phi=(\phi_K)_{K\in\T}$, we introduce the discrete free energy dissipation
\begin{align}\label{3.D}
  \mathcal D[u,\phi] := \sum_{\sigma=K|L\in\E}
  \mathcal P_{K,\sigma}[u,\phi] = \sum_{\sigma=K|L\in\E}
  \F_{K,\sigma}[u,\phi](\D_{K,\sigma}\log u+\D_{K,\sigma}\phi).
\end{align}

\begin{proposition}[Discrete free energy inequality]\label{prop.dei}
Let $(N^k,P^k,Q^k,V^k)_{k=0,\ldots,N_T}$ be the solution to \eqref{3.NPQ}--\eqref{3.bc} constructed in Theorem \ref{thm.ex}. Then, for any $k=1,\ldots,N_T$,
\begin{align*}
  \frac{1}{\Delta t}(\mathcal E_{\rm free}^k 
  - \mathcal E^{k-1}_{\rm free})
  &+ \frac12 \big(\mathcal D[N^k,-V^k] + \mathcal D[P^k,V^k] 
  + \mathcal D[Q^k,V^k]\big) \\
  &\le C\sum_{K\in\T}\m(K)(N_K^k+P_K^k+Q_K^k),
\end{align*}
where the constant $C>0$ depends only on $\|\na(\log N^D-V^D)\|_{L^\infty(\Omega)}$ and $\|\na(\log P^D+V^D)\|_{L^\infty(\Omega)}$, but it is independent of the mesh size $\eta$.
\end{proposition}

Finally, we state the convergence result. Let $\mathcal M_m=(\T_m,\E_m,\mathcal X_m;\Delta t_m,\Delta x_m)$ be a sequence of admissible meshes with mesh size $\eta_m=\max\{\Delta t_m,\Delta x_m\}\to 0$ as $m\to\infty$ satisfying the mesh regularity \eqref{2.meshreg} uniformly in $m\in\N$. Moreover, let $(N,P,Q,V)=(N^k,P^k,Q^k,V^k)_{k=0,\ldots,N_T}$ be the finite-volume solution to \eqref{3.NPQ}--\eqref{3.bc} constructed in Theorem \ref{thm.ex}. To simplify the notation, we set
\begin{align*}
  & N_m := \pi_{\eta_m}N, \
  V_m := \pi_{\eta_m}V, \
  \na_m\sqrt{N} := \pi_{\eta_m}^*(\na_\sigma^{\Delta x_m}\sqrt{N}), \
  \na_mV := \pi_{\eta_m}^*(\na_\sigma^{\Delta x_m}V)
\end{align*}
and analogous notations for $P$ and $Q$. We introduce the Sobolev spaces
\begin{align*}
  H^1_D(\Omega) := \{u\in H^1(\Omega):u=0\mbox{ on }\Gamma_D\}, \quad
  W^{1,1}_D(\Omega) := \{u\in W^{1,1}(\Omega):u=0\mbox{ on }\Gamma_D\},
\end{align*}
and we set $\Omega_T=(0,T)\times\Omega$. 

\begin{theorem}[Convergence of the scheme]\label{thm.conv}
Let Hypotheses (H1)--(H4) hold. There exist limit functions $(N^*,P^*,Q^*,V^*)$ such that, up to a subsequence,
\begin{align*}
  N_m\to N^*,\quad P_m\to P^*,\quad Q_m\to Q^* 
  &\quad\mbox{strongly in }L^1(\Omega_T), \\
  \na_m\sqrt{N}\rightharpoonup\na\sqrt{N^*}, \
  \na_m\sqrt{P}\rightharpoonup\na\sqrt{P^*}, \
  \na_m\sqrt{Q}\rightharpoonup\na\sqrt{Q^*}
  &\quad\mbox{weakly in }L^1(\Omega_T;\R^d),
\end{align*}
as well as $V_m\rightharpoonup V^* $ weakly* in $L^\infty(0,T;H^1(\Omega))$. The quadruple $(N^*,P^*,Q^*,V^*)$ is a weak solution to \eqref{1.eq} in the sense of Definition \ref{def.weak} below.
\end{theorem}

\begin{definition}\label{def.weak}
A quadruple $(N,P,Q,V)$ is called a {\em weak solution} to \eqref{1.eq}--\eqref{1.bic} on $\Omega_T$ if 
\begin{itemize}
\item[\rm (i)] Nonnegativity and regularity: $N$, $P$, $Q\ge 0$ a.e.\ in $\Omega_T$, $N$, $P$, $Q\in L^\infty(0,T;L^1(\Omega))$, $V\in L^\infty(0,T;H^1(\Omega))$ and
\begin{align*}
  \sqrt{N}-\sqrt{N^D},\ \sqrt{P}-\sqrt{P^D}
  \in L^1(0,T;W^{1,1}_D(\Omega)), \quad
  V-V^D\in L^\infty(0,T;H^1_D(\Omega)).
\end{align*}
\item[\rm (ii)] Flux integrability: $J_N:=\na N-N\na V$, $J_P:=\na P+P\na V$, $J_Q:=\na Q+Q\na V\in L^2(0,T;L^1(\Omega))$.
\item[\rm (iii)] Weak formulation: It holds for $\varphi\in C_0^\infty([0,T)\times(\Omega\cup\Gamma_N))$, $\phi\in C^\infty([0,T)\times\Omega)$, $\psi\in H^1_D(\Omega)$, and a.e.\ $t\in(0,T)$ that
\begin{align*}
  -\int_0^T\int_\Omega N\pa_t\varphi dxdt
  - \int_\Omega N^{\rm in}\varphi(0,x)dx
  + \int_0^T\int_\Omega J_N\cdot\na\varphi dxdt &= 0, \\
  -\int_0^T\int_\Omega P\pa_t\varphi dxdt
  - \int_\Omega P^{\rm in}\varphi(0,x)dx
  + \int_0^T\int_\Omega J_P\cdot\na\varphi dxdt &= 0, \\
  -\int_0^T\int_\Omega Q\pa_t\phi dxdt
  - \int_\Omega Q^{\rm in}\phi(0,x)dx
  + \int_0^T\int_\Omega J_Q\cdot\na\phi dxdt &= 0, \\
  \lambda^2\int_\Omega\na V(t)\cdot\na\psi dx
  + \int_\Omega(N(t)-P(t)-Q(t)+A(x))\psi dx &= 0.
\end{align*}
\end{itemize}
\end{definition}


\section{Dissipation structure of the Scharfetter--Gummel flux}
\label{sec.SG}

In this section, we consider a Scharfetter--Gummel discretization of the flux $J$ associated to a general drift-diffusion equation of the type
\begin{align*}
  \pa_t u = \diver J, \quad J=\na u+u\na\phi 
  \quad\mbox{in }\Omega,\ t>0,
\end{align*}
where $\phi$ is a given potential. The energy dissipation $\mathcal D$ associated with the free energy $E=\int_\Omega(u(\log u-1)+u\phi)dx$ is determined by 
\begin{align}\label{2.p}
  \mathcal D = -\frac{dE}{dt} 
  = \int_\Omega J\cdot\na(\log u+\phi)dx 
  = \int_\Omega u|\na(\log u+\phi)|^2 dx \ge 0.
\end{align}
We aim to reproduce this structure in the discrete setting.

\subsection{Scharfetter--Gummel flux and edgewise entropy production}
\label{sec.flux}

The Bernoulli function $B(s):\R\to\R^+$, defined by $B(s)=s/(e^s-1)$ for $s\neq 0$ and $B(0)=1$, is convex, strictly decreasing, and it satisfies
\begin{align*}
  B(s)\ge\max\{0,-s\}\quad\mbox{for }s\in\R, \quad 
  1-s/2\le B(s)\le 1\quad\mbox{for }s\ge 0.
\end{align*}
For a discrete density $u = (u_K)_{K\in\T}$ with $u_K>0$ and a discrete potential $\phi=(\phi_K)_{K\in\T}$, we define the Scharfetter--Gummel flux
\begin{align}\label{2.flux1}
  \F_{K,\sigma}[u,\phi] := \tau_\sigma\big(B(-\D_{K,\sigma}\phi)
  u_{K,\sigma} - B(\D_{K,\sigma}\phi)u_K\big)
  \quad\mbox{for }\sigma\in\E_K.
\end{align}
This expression can be rewritten. Indeed, it follows from \cite[(3.5)--(3.7)]{CCFG21} that
\begin{equation}\label{2.flux2}
\begin{aligned}
  & \F_{K,\sigma}[u,\phi] = \tau_\sigma\widetilde{u}_\sigma
  (\D_{K,\sigma}\log u + \D_{K,\sigma}\phi), \quad\mbox{where} \\
  & \widetilde{u}_\sigma = a_1u_K + a_2u_{K,\sigma}
  \quad\mbox{with }a_1,\,a_2>0,\ a_1+a_2=1,  
\end{aligned}
\end{equation}
and $a_1$ and $a_2$ are given by
\begin{align*}
  a_1 = \frac{B(y)-B(x)}{x-y}, \quad a_2 = \frac{B(-x)-B(-y)}{x-y},
  \quad x:=\D_{K,\sigma}\log u, \quad y:=-\D_{K,\sigma}\phi.
\end{align*}

We now identify the edgewise coercive structure hidden in the entropy production of the Scharfetter--Gummel flux. Recalling definition \eqref{3.D} of the energy dissipation, we deduce from \eqref{2.flux2} that
\begin{align*}
  \mathcal{P}_{K,\sigma}[u,\phi]
  = \tau_\sigma\widetilde{u}_\sigma
  (\D_{K,\sigma}\log u+\D_{K,\sigma}\phi)^2 \ge 0.
\end{align*}
This expression does not directly yield a coercive bound on $\D_{K,\sigma}\sqrt{u}$. The difficulty is that, because of
\begin{align}\label{2.weakchain}
  u_K^{1/2}\D_{K,\sigma}\log u \le 2\D_{K,\sigma}u^{1/2}
  \le u_{K,\sigma}^{1/2}\D_{K,\sigma}\log u,
\end{align}
the edge density $\widetilde{u}_\sigma$ may approach either $u_K$ or $u_L$ as the discrete potential gradient becomes large. (Inequalities \eqref{2.weakchain} follow from
\begin{align}\label{2.ab}
  \min\bigg\{\frac{1}{\sqrt{a}},\frac{1}{\sqrt{b}}\bigg\}
  \le \frac{\log\sqrt{a}-\log\sqrt{b}}{\sqrt{a}-\sqrt{b}}
  \le \max\bigg\{\frac{1}{\sqrt{a}},\frac{1}{\sqrt{b}}\bigg\},
\end{align}
after distinguishing the cases $a\le b$ and $a>b$.) To overcome this issue, we formulate \eqref{2.flux2} in the next subsection in an upwind form that separates the diffusive and drift components.


\subsection{Upwind representation and refined dissipative structure}

First, we show the following upwind form of the numerical flux.

\begin{lemma}[Upwind representation of the Scharfetter--Gummel flux]
\label{lem.upwind}
For every $\sigma\in\E_K$, the Scharfetter--Gummel flux \eqref{2.flux1}  admits the representation
\begin{align}\label{2.flux3}
  \F_{K,\sigma}[u,\phi] = \tau_\sigma B(|\D_{K,\sigma}\phi|)
  \D_{K,\sigma}u + \tau_\sigma\widehat{u}_\sigma
  \D_{K,\sigma}\phi,
\end{align}
where the upwind edge value is defined by
\begin{align}\label{2.upwind}
  \widehat{u}_\sigma = \widehat{u}_\sigma[\phi] 
  := \begin{cases}
  u_{K,\sigma} &\mbox{if }\D_{K,\sigma}\phi\ge 0, \\
  u_K &\mbox{if }\D_{K,\sigma}\phi < 0.
\end{cases}
\end{align}
\end{lemma}

\begin{proof}
The identity $B(s)-B(-s)=-s$ for $s\in\R$ implies that
\begin{align*}
  B(\pm\D_{K,\sigma}\phi) = B(|\D_{\sigma,K}\phi|) 
  + [\D_{K,\sigma}\phi]^{\mp},
\end{align*}
where $[s]^+ = \max\{0,s\}$ and $[s]^- = -\min\{0,s\}$. Substituting these expressions into the Scharfetter--Gummel flux \eqref{2.flux1}, we obtain
\begin{align*}
  \F_{K,\sigma}[u,\phi] = \tau_\sigma
  B(|\D_{K,\sigma}\phi|)(u_{K,\sigma}-u_K)
  + \tau_\sigma
  \big([\D_{K,\sigma}\phi^n]^+u_{K,\sigma}
  - [\D_{K,\sigma}\phi]^-u_K\big).
\end{align*}
We deduce from $[\D_{K,\sigma}\phi]^+ u_{K,\sigma} - [\D_{K,\sigma}\phi]^- u_K = 
\widehat{u}_\sigma\D_{K,\sigma}\phi$ that formulation \eqref{2.flux3} holds.
\end{proof}

Substituting formulation \eqref{2.flux3} of the Scharfetter--Gummel flux into definition \eqref{3.D} of the entropy production, the upwind form of the edgewise entropy production follows:
\begin{align}\label{2.prod2}
  \mathcal{P}_{K\sigma}[u,\phi] = \tau_\sigma
  \big(B(|\D_{K,\sigma}\phi|)\D_{K,\sigma}u 
  + \widehat{u}_\sigma\D_{K,\sigma}\phi\big)
  \D_{K,\sigma}(\log u + \phi).
\end{align}
This formulation contains an explicit Fisher-information-type contribution, up to some cross term, as shown in the following proposition.

\begin{proposition}[Edgewise coercive lower bound]\label{prop.coerc}
For every edge $\sigma\in\E_K$ and every cellwise positive density $u$, the edgewise entropy production \eqref{2.prod2} satisfies the bound
\begin{align}\label{2.prodct}
  \mathcal{P}_{K,\sigma}[u,\phi] \ge 4\tau_\sigma
  B(|\D_{K,\sigma}\phi|)|\D_{K,\sigma}\sqrt{u}|^2
  + \frac{\tau_\sigma}{2}\widehat{u}_\sigma
  |\D_{K,\sigma}\phi|^2 + 2\tau_\sigma\D_{K,\sigma}u
  \D_{K,\sigma}\phi,
\end{align}
recalling definition \eqref{2.upwind} of the upwind edge value $\widehat{u}_\sigma$.
\end{proposition}

\begin{proof}
We deduce from the elementary inequality $(\log a-\log b)(a-b)\ge 4(\sqrt{s}-\sqrt{b})^2$ for $a$, $b>0$ that $ \D_{K,\sigma}u\D_{K,\sigma}\log u\ge 4|\D_{K,\sigma}\sqrt{u}|^2$.
Using inequalities \eqref{2.ab} with $\sqrt{a}=u_L$ and $\sqrt{b}=u_K$ as well as definition \eqref{2.upwind} of $\widehat{u}_\sigma$, we have
$\widehat{u}_\sigma\D_{K,\sigma}\phi\D_{K,\sigma}\log u \ge \D_{K,\sigma}u\D_{K,\sigma}\phi$. Substituting these estimates into \eqref{2.prod2} leads to
\begin{align}\label{2.aux}
  \mathcal{P}_{K,\sigma}[u,\phi]
  &\ge 4\tau_\sigma B(|\D_{K,\sigma}\phi|)|\D_{K,\sigma}\sqrt{u}|^2
  + \tau_\sigma\widehat{u}_\sigma|\D_{K,\sigma}\phi|^2 \\
  &\phantom{xx}+ \tau_\sigma\big(1+B(|\D_{K,\sigma}\phi|)\big)
  \D_{K,\sigma}u\D_{K,\sigma}\phi. \nonumber 
\end{align}
We write the last term as
\begin{align}\label{2.aux2}
  \tau_\sigma\big(1+B(|\D_{K,\sigma}\phi|\big)
  \D_{K,\sigma}u\D_{K,\sigma}\phi
  &= 2\tau_\sigma\D_{K,\sigma}u\D_{K,\sigma}\phi \\
  &\phantom{xx}+ \tau_\sigma\big(B(|\D_{K,\sigma}\phi|)-1\big)
  \D_{K,\sigma}u\D_{K,\sigma}\phi. \nonumber 
\end{align}
It follows from $0\le 1-B(s)\le s/2$ and 
\begin{align*}
  -\D_{K,\sigma}u\D_{K,\sigma}\phi
  \ge -u_{K,\sigma}[\D_{K,\sigma}\phi]^+
  - u_K[\D_{K,\sigma}\phi]^-
\end{align*}
that
\begin{align*}
  \tau_\sigma\big(B(|\D_{K,\sigma}\phi|)-1\big)
  \D_{K,\sigma}u\D_{K,\sigma}\phi
  \ge -\frac{\tau_\sigma}{2}\widehat{u}_\sigma
  |\D_{K,\sigma}\phi|^2.
\end{align*}
Inserting this estimate into \eqref{2.aux2} and then into \eqref{2.aux} shows the result.
\end{proof}

The following result provides an additional Fisher information–type formulation of the entropy production written as a square.

\begin{proposition}[Refined dissipation square identity]
\label{prop.square}
For every edge $\sigma\in\E_K$ and every cellwise positive density $u$, the edgewise entropy production \eqref{2.prod2} is written as
\begin{align*}
  \mathcal{P}_{K,\sigma}[u,\phi] 
  = \tau_\sigma^*\bigg(2B(|\D_{K,\sigma}\phi|)\D_{K,\sigma}\sqrt{u}
  + \frac{\widehat{u}_\sigma}{\overline{u}_\sigma^{1/2}}
  \D_{K,\sigma}\phi\bigg)^2,
\end{align*}
where the upwind edge value $\widehat{u}_\sigma$ is defined in \eqref{2.upwind}, and the $\frac12$-mean $\overline{u}_\sigma$ and $\tau_\sigma^*$ are given by
\begin{align}\label{2.prodsq}
  \overline{u}_\sigma = \bigg(\frac{u_K^{1/2}
  + u_L^{1/2}}{2}\bigg)^2, \quad
  \tau_\sigma^* := \tau_\sigma\frac{\overline{u}_\sigma}{
  \widetilde{u}_\sigma} \ge \frac{\tau_\sigma}{4},
\end{align}
recalling definition \eqref{2.flux2} of $\widetilde{u}_\sigma$. 
\end{proposition}

\begin{proof}
Taking into account formulations \eqref{2.flux2} and \eqref{2.flux3} of the numerical flux, the entropy production can be written as
\begin{align}\label{2.P}
  \mathcal{P}_{K,\sigma}[u,\phi] 
  = \frac{1}{\tau_\sigma\widetilde{u}_\sigma}
  \big(\F_{K,\sigma}[u,\phi]\big)^2
  = \frac{\tau_\sigma}{\widetilde{u}_\sigma}
  \big(B(|\D_{K,\sigma}\phi|)\D_{K,\sigma}u 
  + \widehat{u}_\sigma \D_{K,\sigma}\phi\big)^2.
\end{align}
Then, substituting
\begin{align*}
  \D_{K,\sigma}u = \big(u_K^{1/2}+u_L^{1/2}\big)
  \D_{K,\sigma}\sqrt{u} = 2\overline{u}_\sigma^{1/2}
  \D_{K,\sigma}\sqrt{u}
\end{align*}
in the previous expression yields \eqref{2.prodsq}. Finally, it follows from
\begin{align*}
  \overline{u}_\sigma = \bigg(\frac{u_K^{1/2}
  + u_L^{1/2}}{2}\bigg)^2 \ge \frac{u_K+u_L}{4}
  \ge \frac{\widetilde{u}_\sigma}{4}
\end{align*}
that $\tau_\sigma^* = \tau_\sigma\overline{u}_\sigma/ \widetilde{u}_\sigma \ge \tau_\sigma/4$, finishing the proof.
\end{proof}


\subsection{Complementary edgewise coercivity estimates}

Estimates \eqref{2.prodct} and \eqref{2.prodsq} of the entropy production degenerate because of $B(s)\to 0$ as $s\to\infty$. Consequently, neither estimate alone yields a direct control of $\D_{K,\sigma}\sqrt{u}$. In the following, we derive two complementary edgewise estimates that provide, together with \eqref{2.prodct} and \eqref{2.prodsq}, a control of $\D_{K,\sigma}\sqrt{u}$. To this end, we introduce the symmetric cross term $\mathcal{C}_{K,\sigma}$ and the upwind cross term $\mathcal{C}_{K,\sigma}^{\rm up}$:
\begin{align}
  \mathcal C_{K,\sigma}[u,\phi]
  &:= \tau_\sigma D_{K,\sigma}u\D_{K,\sigma}\phi
  = 2\tau_\sigma\overline u_\sigma^{1/2}
  \D_{K,\sigma}\sqrt{u}\D_{K,\sigma}\phi, \label{2.ct} \\
  \mathcal C^{\rm up}_{K,\sigma}[u,\phi]
  &:= 2\tau_\sigma\widehat u_\sigma^{1/2}
  \D_{K,\sigma}\sqrt{u}\D_{K,\sigma}\phi. \label{2.ctup}
\end{align}
The following identity quantifies the gap between both cross terms.

\begin{lemma}[Coupling identity]\label{lem.coup}
For any cellwise positive density $u$ and potential $\phi$,
\begin{align*}
  \mathcal C^{\rm up}_{K,\sigma}[u,\phi]
  - \mathcal C_{K,\sigma}[u,\phi]
  = \tau_\sigma|\D_{K,\sigma}\phi||\D_{K,\sigma}\sqrt{u}|^2.
\end{align*}
\end{lemma}

\begin{proof}
If $\D_{K,\sigma}\phi\ge 0$ then $\widehat{u}_\sigma^{1/2} = 
u_{K,\sigma}^{1/2}$ and
\begin{align*}
  2\widehat{u}_\sigma^{1/2}\D_{K,\sigma}\phi
  = |\D_{K,\sigma}\phi|\big(u_{K,\sigma}^{1/2} - u_K^{1/2}\big)
  + \D_{K,\sigma}\phi
  \big(u_{K,\sigma}^{1/2} + u_K^{1/2}\big).
\end{align*}
If $\D_{K,\sigma}\phi<0$, the same identity holds with $u_K$ and $u_{K,\sigma}$ interchanged. After multiplication by $\D_{K,\sigma}\sqrt{u}$, this yields
\begin{align*}
  \mathcal C^{\rm up}_{K,\sigma}[u,\phi]
  &= \tau_\sigma\big(\D_\sigma\phi|\D_{K,\sigma}\sqrt{u}|^2
  + 2\overline{u}_{K,\sigma}^{1/2}\D_{K,\sigma}\phi
  \D_{K,\sigma}\sqrt{u}\big) \\
  &= \tau_\sigma\big(\D_\sigma\phi|\D_{K,\sigma}\sqrt{u}|^2
  + \mathcal C_{K,\sigma}[u,\phi]\big),
\end{align*}
showing the result.
\end{proof}

Our first coercivity bound involves the cross term \eqref{2.ct}.

\begin{proposition}[First edgewise coercivity estimate]\label{prop.first}
It holds for every edge $\sigma\in\E_K$ and every cellwise positive density $u$ and potential $\phi$ that
\begin{align*}
  \tau_\sigma |\D_{K,\sigma}\sqrt{u}|^2
  \lesssim\mathcal P_{K,\sigma}[u,\phi]
  - \mathcal C_{K,\sigma}[u,\phi].
\end{align*}
\end{proposition} 

\begin{proof}
Since $0\le 1-B(s)\le s/2$ for all $s\ge0$, we can estimate
\begin{align}\label{2.aux3}
  \tau_\sigma|D_{K,\sigma}\sqrt{u}|^2
  &= \tau_\sigma B(|\D_{K,\sigma}\phi|)|\D_{K,\sigma}\sqrt{u}|^2
  + \tau_\sigma \big(1 - B(|\D_{K,\sigma}\phi|)\big)
  |\D_{K,\sigma}\sqrt{u}|^2 \\
  &\le \tau_\sigma B(|\D_{K,\sigma}\phi|)|\D_{K,\sigma}\sqrt{u}|^2
  + \frac{\tau_\sigma}{2} 
  |\D_{K,\sigma}\phi||\D_{K,\sigma}\sqrt{u}|^2. \nonumber 
\end{align}
By Lemma \ref{lem.coup} and Young's inequality, the last term becomes
\begin{align*}
  \frac{\tau_\sigma}{2} 
  |\D_{K,\sigma}\phi||\D_{K,\sigma}\sqrt{u}|^2
  &= \frac12\mathcal C^{\rm up}_{K,\sigma}[u,\phi]
  - \frac12\mathcal C_{K,\sigma}[u,\phi] \\
  &\le \frac{\tau_\sigma}{2}|\D_{K,\sigma}\sqrt{u}|^2
  + \frac{\tau_\sigma}{2}\widehat u_\sigma|\D_{K,\sigma}\phi|^2
  - \frac12\mathcal C_{K,\sigma}[u,\phi].
\end{align*}
Substituting this expression into \eqref{2.aux3}, rearranging the terms, and then using Proposition \ref{prop.coerc} yields
\begin{align*}
  \frac{\tau_\sigma}{2}|D_{K,\sigma}\sqrt{u}|^2
  &\le \tau_\sigma B(|\D_{K,\sigma}\phi|)|\D_{K,\sigma}\sqrt{u}|^2
  + \frac{\tau_\sigma}{2}\widehat u_\sigma|\D_{K,\sigma}\phi|^2
  - \frac12\mathcal C_{K,\sigma}[u,\phi] \\
  &\le 4\tau_\sigma B(|\D_{K,\sigma}\phi|)|\D_{K,\sigma}\sqrt{u}|^2
  + \frac{\tau_\sigma}{2}\widehat u_\sigma|\D_{K,\sigma}\phi|^2
  - \frac12\mathcal C_{K,\sigma}[u,\phi] \\
  &\le \mathcal{P}_{K,\sigma}[u,\phi] 
  - \frac52\mathcal C_{K,\sigma}[u,\phi],
\end{align*}
which finishes the proof.
\end{proof}

Our second estimate is weaker, but it is free of the cross term.

\begin{proposition}[Second edgewise coercivity estimate]
\label{prop.second}
It holds for every edge $\sigma\in\E_K$ and every cellwise positive density $u$ and potential $\phi$ that
\begin{align*}
  \m(\sigma)|\D_{K,\sigma}\sqrt{u}|
  \lesssim \mathcal P_{K,\sigma}[u,\phi] 
  + \tau_\sigma|\D_{K,\sigma}\phi|^2 
  + \m(\Delta_\sigma)(1+\overline{u}_\sigma).
\end{align*}
\end{proposition}

\begin{proof}
We decompose and use $0\le 1-B(s)\le s/2$ for $s\ge 0$ as well as Proposition \ref{prop.square}:
\begin{align*}
  2|\D_{K,\sigma}\sqrt{u}|
  &= \bigg| 2\big(1-B(|\D_{K,\sigma}\phi|)\big)\D_{K,\sigma}\sqrt{u}
  - \frac{\widehat u_\sigma}{\overline u_\sigma^{1/2}}
  \D_{K,\sigma}\phi \\
  &\phantom{xx} + \bigg(2B(|\D_{K,\sigma}\phi|)\D_{K,\sigma}\sqrt{u}
  + \frac{\widehat u_\sigma}{\overline u_\sigma^{1/2}}
  \D_{K,\sigma}\phi\bigg)\bigg| \\
  &\le |\D_{K,\sigma}\phi||\D_{K,\sigma}\sqrt{u}|
  + \frac{\widehat u_\sigma}{\overline u_\sigma^{1/2}}
  |\D_{K,\sigma}\phi|
  + \big(4\tau_\sigma^{-1}\mathcal{P}_{K,\sigma}[u,\phi]\big)^{1/2}.
\end{align*}
It follows from $|\D_{K,\sigma}\sqrt{u}|\le 2\overline u_\sigma^{1/2}$ and $\widehat u_\sigma \le \max\{u_{K,\sigma},u_K\} \le (u_{K,\sigma}^{1/2}+u_K^{1/2})^2 = 4\overline u_\sigma$ that
\begin{align*}
  2|\D_{K,\sigma}\sqrt{u}|
  \le 6|\D_{K,\sigma}\phi|\overline u_\sigma^{1/2}
  + \big(4\tau_\sigma^{-1}\mathcal{P}_{K,\sigma}[u,\phi]\big)^{1/2}.
\end{align*}
We multiply this inequality by $\m(\sigma)/2$, use relation \eqref{2.mdelta} together with $\tau_\sigma=\m(\sigma)/\dist_\sigma$, and apply Young's inequality:
\begin{align*}
  \m(\sigma)|\D_{K,\sigma}\sqrt{u}|
  &\le 3\m(\sigma)|\D_{K,\sigma}\phi|\overline u_\sigma^{1/2}
  + \big(d\m(\Delta_\sigma)\mathcal{P}_{K,\sigma}[u,\phi]
  \big)^{1/2} \\
  &\le \frac32\tau_\sigma|\D_{K,\sigma}\phi|^2
  + \frac32\m(\sigma)\dist_\sigma\overline u_\sigma
  + \frac12\mathcal P_{K,\sigma}[u,\phi] 
  + \frac{d}{2}\m(\Delta_\sigma).
\end{align*}
An application of identity \eqref{2.mdelta}, namely $\m(\sigma)\dist_\sigma = d\m(\Delta_\sigma)$, finishes the proof.
\end{proof}


\section{Existence of discrete solutions and global a priori bounds}
\label{sec.ex}

In this section, we establish the existence of a solution to the finite-volume scheme. We start with the proof of the discrete free energy inequality.

\subsection{Proof of Proposition \ref{prop.dei}}

Let $(N^k,P^k,Q^k,V^k)_{k=0,\ldots,N_T}$ be a discrete solution with cellwise positive discrete densities. 
We introduce the relative chemical potentials
\begin{align*}
  & \mu_{N,K}^k := \log\frac{N_K^k}{N_K^D} - (V_K^k-V_K^D), \quad
  \mu_{P,K}^k := \log\frac{P_K^k}{P_K^D} + (V_K^k-V_K^D), \\
  & \mu_{Q,K}^k := \log Q_K^k + V_K ^k
  = \log\frac{Q_K^k}{Q_K^D} + (V_K^k-V_K^D),
\end{align*}
recalling that $Q^D=\exp(-V^D)$. We multiply equations \eqref{3.NPQ} by $\mu_{N,K}^k$, $\mu_{P,K}^k$, and $\mu_{Q,K}^k$, respectively, sum over $K\in\T$, and use the discrete integration-by-parts formula \eqref{2.dibp}. The equation for $N^k$ becomes $I_1+I_2+I_3=0$, where
\begin{align*}
  I_1 &= \frac{1}{\Delta t}\sum_{K\in\T}\m(K)(N_K^k-N_K^{k-1})
  \log\frac{N_K^k}{N_K^D}, \\
  I_2 &= -\frac{1}{\Delta t}\sum_{K\in\T}\m(K)(N_K^k-N_K^{k-1})
  (V_K^k-V_K^D), \\
  I_3 &= \sum_{\sigma=K|L\in\E}\F_{K,\sigma}[N^k,-V^k]
  \D_{K,\sigma}\bigg(\log\frac{N^k}{N^D}-(V^k-V^D)\bigg).
\end{align*}
We deduce from the convexity of $s\mapsto s\log s$ that $(a-b)\log a\ge a(\log a-1)-b(\log b-1)$ and hence,
\begin{align*}
  (N_K^k-N_K^{k-1})&\log\frac{N_K^k}{N_K^D} 
  \ge N_K^k(\log N_K^k-1) - N_K^{k-1}(\log N_K^{k-1}-1)
  - (N_K^k-N_K^{k-1})\log N_K^D \\
  &= \bigg(N_K^k\log\frac{N_K^k}{N_K^D}-(N_K^k-N_K^D)\bigg)
  - \bigg(N_K^{k-1}\log\frac{N_K^{k-1}}{N_K^D}-(N_K^{k-1}-N_K^D)\bigg).
\end{align*}
Consequently, by definition \eqref{3.H} of the relative entropy,
\begin{align*}
  I_1 \ge \frac{1}{\Delta t}\big(\mathcal H(N^k|N^D)
  - \mathcal H(N^{k-1}|N^D)\big).
\end{align*}
The term $I_2$ will be considered later. Definition \eqref{3.D} of the free energy dissipation, formulation \eqref{2.flux2} of the flux $\F_{K,\sigma}[N^k,-V^k]$, and Young's inequality yield
\begin{align*}
  I_3 &= \mathcal D[N^k,-V^k] + \sum_{\sigma=K|L\in\E}
  \F_{K,\sigma}[N^k,-V^k]\D_{K,\sigma}(\log N^D-V^D) \\
  &\ge \frac12\mathcal D[N^k,-V^k] - \frac12\sum_{\sigma=K|L\in\E}
  \tau_\sigma\widetilde{N}^k_\sigma|\D_{K,\sigma}(\log N^D-V^D)|^2.
\end{align*}

To bound the last term, we use the following mesh estimate. For any $\psi\in W^{1,\infty}(\Omega)$, there exists a constant $C>0$, independent of the mesh parameters, such that
\begin{align*}
  \tau_\sigma|\D_{K,\sigma}\psi|^2 \le C\|\na\psi\|_{L^\infty(\Omega)}^2
  (\m(K)+\m(L)) \quad\mbox{for }\sigma=K|L\in\E.
\end{align*}
Applying this estimate with $\psi=\log N^D-V^D$ and using the inequality $\widetilde{N}_\sigma^k\le N_L^k+N_K^k$, we conclude that 
\begin{align*}
  I_3 \ge \frac12\mathcal D[N^k,-V^k] - C\sum_{K\in\T}\m(K)N_K^k
\end{align*}
for some $C>0$, depending on the $L^\infty(\Omega)$ norm of $\log N^D-V^D$. 

Summarizing these estimations, we obtain
\begin{align*}
  \frac{1}{\Delta t}&\big(\mathcal H(N^k|N^D)
  - \mathcal H(N^{k-1}|N^D)\big)
  + \frac12\mathcal D[N^k,-V^k] \\
  &\le \frac{1}{\Delta t}\sum_{K\in\T}\m(K)(N_K^k-N_K^{k-1})
  (V_K^k-V_K^D) + C\sum_{K\in\T}\m(K)N_K^k.
\end{align*}
We apply similar arguments to the equations for $P$ and $Q$ and sum the resulting inequalities:
\begin{align*}
  \frac{1}{\Delta t}&\big(\mathcal H(N^k|N^D) 
  + \mathcal H(P^k|P^D) + \mathcal H(Q^k|Q^D)\big) \\
  &\phantom{xx}+ \frac12\big(\mathcal D[N^k,-V^k] + \mathcal D[P^k,V^k]
  + \mathcal D[Q^k,V^k]\big) \\
  &\le \frac{1}{\Delta t}\sum_{K\in\T}\m(K)\big((N_K^k-P_K^k-Q_K^k)
  - (N_K^{k-1}-P_K^{k-1}-Q_K^{k-1})\big)(V_K^k-V_K^D) \\
  &\phantom{xx}+ C\sum_{K\in\T}\m(K)(N_K^k+P_K^k+Q_K^k).
\end{align*}

For the next-to-last term, we take the difference of the discrete equations \eqref{3.V} for $V^k$ and $V^{k-1}$, multiply the resulting equation by $V_K^k-V_K^{D}$, sum over $K\in\T$, apply the discrete integration-by-parts formula, and finally use Young's inequality:
\begin{align*}
  \sum_{K\in\T}&\m(K)\big((N_K^k-P_K^k-Q_K^k)
  - (N_K^{k-1}-P_K^{k-1}-Q_K^{k-1})\big)(V_K^k-V_K^D) \\
  &= \lambda^2\sum_{K\in\T}\sum_{\sigma\in\E_K}\tau_\sigma
  \D_{K,\sigma}(V^k-V^{k-1})(V_K^k-V_K^D) \\
  &= -\lambda^2\sum_{\sigma=K|L\in\E_K}\tau_\sigma
  \D_{K,\sigma}(V^k-V^{k-1})\D_{K,\sigma}(V_K^k-V_K^D) \\
  &= -\lambda^2\sum_{\sigma=K|L\in\E_K}\tau_\sigma
  \big(\D_{K,\sigma}(V^k-V_K^D) - \D_{K,\sigma}(V^{k-1}-V_K^D)\big)
  \D_{K,\sigma}(V_K^k-V_K^D) \\
  &\le -\frac{\lambda^2}{2}\sum_{\sigma=K|L\in\E_K}\tau_\sigma
  \big|(\D_{K,\sigma}(V^k-V_K^D)\big|^2
  + \frac{\lambda^2}{2}\sum_{\sigma=K|L\in\E_K}\tau_\sigma
  \big|(\D_{K,\sigma}(V^{k-1}-V_K^D)\big|^2 \\
  &= -\mathcal G(V^k;V^D) + \mathcal G(V^{k-1};V^D),
\end{align*}
recalling definition \eqref{3.G} of the electric energy. Taking into account definition \eqref{3.E} of the total energy $\E_{\rm free}^k$, this concludes the proof.


\subsection{Global a priori bounds}

We collect the global a priori bounds. The first lemma gathers the uniform mass, entropy, electric energy, and dissipation bounds, which follow from the discrete free energy inequality. The second lemma combines these bounds with the refined dissipation structure in Section \ref{sec.flux} to control the Fisher information.

\begin{lemma}[Uniform mass, entropy, and dissipation bounds]
\label{lem.est1}
There exists a constant $C>0$, independent of the mesh parameter $\eta$, such that the following bounds hold:
\begin{itemize}
\item[\rm (i)] Discrete masses:
\begin{align*}
  \max_{k=1,\ldots,N_T}\sum_{K\in\T}\m(K)(N_K^k+P_K^k+Q_K^k) \le C.
\end{align*}
\item[\rm (ii)] Discrete entropy and electric energy:
\begin{align*}
  \sum_{K\in\T}\m(K)\big(N_K^k\log N_K^k + P_K^k\log P_K^k
  + Q_K^k\log Q_K^k\big)
  + \sum_{\sigma\in\E}\m(\Delta_\sigma)|\na_\sigma^{\Delta x}V^k|^2
  \le C.
\end{align*}
\item[\rm (iii)] Dissipation terms:
\begin{align*}
  \sum_{k=1}^{N_T}\Delta t\big(\mathcal D[N^k,-V^k] 
  + \mathcal D[P^k,V^k] + \mathcal D[Q^k,V^k]\big) \le C.
\end{align*}
\end{itemize}
\end{lemma}

\begin{proof}
The estimates (ii) and (iii) are direct consequences of the discrete free energy inequality after summation over $k=1,\ldots,N_T$. Then bound (ii) implies the control of the total masses, proving estimate (i).
\end{proof}

Recall definition \eqref{2.upwind} of the upwind density $\widehat{u}_\sigma[\phi]$. We set
\begin{align}\label{4.widehatN}
  \widehat N_\sigma^k := \widehat N_\sigma^k(-V^k), \quad
  \widehat P_\sigma^k := \widehat P_\sigma^k(V^k), \quad
  \widehat Q_\sigma^k := \widehat Q_\sigma^k(V^k).
\end{align} 

\begin{lemma}[Uniform refined-flux and $L^1$-gradient bounds]
\label{lem.flux}
There exists a constant $C>0$, independent of the mesh parameter $\eta$, such that
\begin{align}
  \sum_{k=1}^{N_T}\Delta t\sum_{\sigma\in\E}\m(\Delta_\sigma)
  \bigg|B(|\D_{K,\sigma}V^k|)\na_\sigma^{\Delta x}(N^k)^{1/2}
  - \frac{\widehat N_\sigma^k}{(\overline N_\sigma^k)^{1/2}}\bigg|^2
  &\le C, \nonumber \\
  \sum_{k=1}^{N_T}\Delta t\sum_{\sigma\in\E}\m(\Delta_\sigma)
  \bigg|B(|\D_{K,\sigma}V^k|)\na_\sigma^{\Delta x}(P^k)^{1/2}
  + \frac{\widehat P_\sigma^k}{(\overline P_\sigma^k)^{1/2}}\bigg|^2
  &\le C, \label{4.dissP} \\
  \sum_{k=1}^{N_T}\Delta t\sum_{\sigma\in\E}\m(\Delta_\sigma)
  \bigg|B(|\D_{K,\sigma}V^k|)\na_\sigma^{\Delta x}(Q^k)^{1/2}
  + \frac{\widehat Q_\sigma^k}{(\overline Q_\sigma^k)^{1/2}}\bigg|^2
  &\le C. \nonumber 
\end{align}
Moreover, we have
\begin{align}\label{4.L1grad}
  \sum_{k=1}^{N_T}\Delta t\sum_{\sigma\in\E}\m(\Delta_\sigma)
  \big(|\na_\sigma^{\Delta x}(N^k)^{1/2}|
  + |\na_\sigma^{\Delta x}(P^k)^{1/2}|
  + |\na_\sigma^{\Delta x}(Q^k)^{1/2}|\big) \le C.
\end{align}
\end{lemma}

\begin{proof}
We use the refined dissipation identity from Proposition \ref{prop.square} and the discrete mass bounds from Lemma \ref{lem.est1} to infer from Proposition \ref{prop.dei} for the pair $(P,V)$ that
\begin{align*}
  \sum_{k=1}^{N_T}\Delta t\sum_{\sigma\in\E}\tau_\sigma
  \bigg|2B(|\D_{K,\sigma}V^k|)\D_{K,\sigma}(P^k)^{1/2}
  + \frac{\widehat P_\sigma^k}{(\overline P_\sigma^k)^{1/2}}
  \D_{K,\sigma} V^k\bigg|^2 \le C.
\end{align*} 
Because of definition \eqref{2.grad} of the discrete gradient, this is equivalent to \eqref{4.dissP}. The estimates for $N$ and $Q$ are derived in the same way. 

It remains to prove \eqref{4.L1grad}. By Proposition \ref{prop.second}, the pair $(P,V)$ satisfies
\begin{align}\label{4.aux} 
  \sum_{\sigma\in\E}\m(\sigma)|\D_{K,\sigma}(P^k)^{1/2}|
  \le C\sum_{\sigma\in\E}\big(\mathcal P_{K,\sigma}[P^k,V^k]
  + \tau_\sigma|\D_{K,\sigma}V^k|^2
  + \m(\Delta_\sigma)(1+\overline{P}_\sigma^k)\big).
\end{align}
The last term is estimated by using $\overline P_\sigma\le \frac12(P_K^k+P_{K,\sigma}^k)$ and $\m(\Delta_\sigma)\le C\m(K)$ (which follows from the mesh regularity; see Lemma \ref{lem.meshregul} in Appendix \ref{sec.app}):
\begin{align*}
  \sum_{\sigma\in\E}\m(\Delta_\sigma)(1+\overline{P}_\sigma^k)
  \le C\sum_{K\in\T}\m(K)(1+P_K^k).
\end{align*}
Hence, summing \eqref{4.aux} over $k=1,\ldots,N_T$ and taking into account that the discrete total mass is bounded by Lemma \ref{lem.est1}, we obtain \eqref{4.L1grad}. The estimates for $N^k$ and $Q^k$ follow in the same way by replacing $(P^k,V^k)$ by $(N^k,-V^k)$ and $(Q^k,V^k)$. 
\end{proof}


\subsection{Proof of Theorem \ref{thm.ex}} 

We study first a discrete Fokker--Planck problem with prescribed potential. Let $k\in\{1,\ldots,N_T\}$ and let $u^{k-1}=(u^{k-1}_K)_{K\in\T}\in\R^{\#\mathcal T}$ be a cellwise nonnegative density and let $\phi^k=(\phi_K^k)_{K\in\T}\in\R^{\#\mathcal T}$. We consider the linear problem
\begin{align}\label{4.eq}
   \frac{\m(K)}{\Delta t}(u_K^k-u_K^{k-1})
  - \sum_{\sigma\in\E_K}\F_{K,\sigma}[u^k,\phi^k] = 0
  \quad\mbox{for }K\in\T,
\end{align}
supplemented with the mixed boundary conditions
\begin{align}\label{4.bc}
  u_\sigma^k = u_\sigma^D\quad\mbox{for }\sigma\in\E_{\rm ext}^D, \quad
  \F_{K,\sigma}[u^k,\phi^k] = 0 
  \quad\mbox{for }\sigma\in\E_{{\rm ext},K}^N.
\end{align}
If $\Gamma_D=\emptyset$, the boundary conditions consist of the homogeneous no-flux conditions only. 

\begin{lemma}\label{lem.pos}
Let $u^{k-1}\not\equiv 0$ be nonnegative. Then problem \eqref{4.eq}--\eqref{4.bc} (including the case $\Gamma_D=\emptyset$) admits a unique solution $u^k\in \R^{\#\mathcal T}$ satisfying $u_K^k>0$ for all $K\in\T$.
\end{lemma}

\begin{proof}
The existence of a solution to \eqref{4.eq}--\eqref{4.bc} follows from \cite[Prop.~1]{Bes12}. It remains to prove the positivity. Suppose, by contradiction, that $u_{K_0}^k=0$ for some $K_0\in\T$. Then equation \eqref{4.eq} becomes
\begin{align*}
  \frac{\m(K_0)}{\Delta t}u_{K_0}^{k-1}
  + \sum_{\sigma\in\E_{K_0}\cap\E_{\rm int}}\tau_\sigma
  B(-\D_{K_0,\sigma}\phi^k)u_{K_0,\sigma}^k
  + \sum_{\sigma\in\E_{{\rm ext},K_0}^D}\tau_\sigma
  B(-\D_{K_0,\sigma}\phi^k)u_\sigma^D = 0.
\end{align*}
All terms on the left-hand side are nonnegative. Hence, each term vanishes. In particular, $u_{K_0}^{k-1}=0$ and $u_{K_0,\sigma}^k=0$ for all interior neighboring cells. Repeating this argument for each neighbor and using the connectivity of the mesh, we conclude that $u_K^k=0$ and $u_K^{k-1}=0$ for all $K\in\T$. This contradicts the assumption $u^{k-1}\not\equiv 0$.
\end{proof}

We continue with the proof of Theorem \ref{thm.ex}. Let $k\in\{1,\ldots,N_T\}$ and assume that $(N^{k-1},P^{k-1},Q^{k-1})$ is given and positive. For a given triple $X^*=(N^*,P^*,Q^*)\in(\R^{\#\mathcal T})^3$, let $V^*=V[X^*]\in\R^{\#\mathcal T}$ be the unique solution to the discrete Poisson problem
\begin{align*}
  \lambda^2\sum_{\sigma\in\E_K}\tau_\sigma\D_{K,\sigma}V^*
  = \m(K)(N_K^*-P_K^*-Q_K^*+A_K)\quad\mbox{for }K\in\T,
\end{align*}
with the boundary values prescribed in \eqref{3.bc}. Next, we solve the linear problems
\begin{align*}
  \frac{\m(K)}{\Delta t}(N_K^k-N_K^{k-1}) 
  - \sum_{\sigma\in\E_K}\F_{K,\sigma}[N^k,-V^*] &= 0, \\
  \frac{\m(K)}{\Delta t}(P_K^k-P_K^{k-1}) 
  - \sum_{\sigma\in\E_K}\F_{K,\sigma}[P^k,V^*] &= 0, \\
  \frac{\m(K)}{\Delta t}(Q_K^k-Q_K^{k-1}) 
  - \sum_{\sigma\in\E_K}\F_{K,\sigma}[Q^k,V^*] &= 0
  \quad\mbox{for }K\in\T,
\end{align*}
together with the boundary conditions prescribed in \eqref{3.bc}. This defines the fixed-point mapping $S:(\R^{\#\mathcal T})^3\to (\R^{\#\mathcal T})^3$, $S(X^*) = (N^k,P^k,Q^k)$. By construction, $S$ is continuous. To apply Schaefer's fixed-point theorem, it remains to prove that
\begin{align*}
  \{X\in (\R^{\#\mathcal T})^3: X = \theta S(X)\mbox{ for some }
  \theta\in[0,1]\}
\end{align*}
is bounded. This follows from the a priori estimates established above. By Schaefer's theorem \cite[Sec.~9.2.2]{Eva98}, $S$ admits a fixed point, which is a solution to \eqref{3.NPQ}--\eqref{3.bc}. 


\section{Localized coercivity and interior estimates}\label{sec.chi}

We derive an interior Fisher information estimate using the global bounds established in the previous section and the edgewise coercivity estimate of Proposition \ref{prop.first}. The crucial step is the treatment of the cross terms. Because of the mixed-type boundary conditions, this step cannot be carried out up to the Dirichlet boundary and therefore requires a localization procedure away from $\Gamma_D$. 
Let $\eps>0$ and define
\begin{align*}
  \Omega_D^\eps := \{x\in\overline\Omega:d(x,\Gamma_D)>\eps\},
  \quad \Omega_D^{2\eps} := \{x\in\overline\Omega:
  d(x,\Gamma_D)>2\eps\}.
\end{align*}
Let $\chi^\eps\in C_0^\infty(\Omega_D^\eps)$ be such that $0\le\chi^\eps\le 1$ in $\Omega_D^\eps$ and $\chi^\eps\equiv 1$ in $\Omega_D^{2\eps}$. We set $\chi_K^\eps:=\chi^\eps(x_K)$ for $K\in\T$ and $\chi_\sigma^\eps:=\chi^\eps(x_\sigma)$ for $\sigma\in\E_{\rm ext}$, where $x_\sigma:=\m(\sigma)^{-1}\int_\sigma xdx$, and assume that
\begin{align}\label{5.chieps}
  \frac{|\D_{K,\sigma}\chi^\eps|}{\dist_\sigma}
  \le \frac{C}{\eps}\quad\mbox{for }\sigma\in\E_K.
\end{align}

For a positive discrete density $u$ and a discrete potential $\phi$, we recall the definition of the edgewise entropy production and cross term:
\begin{align*}
  \mathcal P_{K,\sigma}[u,\phi] = \F_{K,\sigma}[u,\phi]
  \D_{K,\sigma}(\log u + \phi), \quad
  \mathcal C_{K,\sigma}[u,\phi] 
  = \tau_\sigma\D_{K,\sigma}u\D_{K,\sigma}\phi.
\end{align*}
It follows from $0\le\chi^\eps_K\le 1$ that
\begin{align*}
  \sum_{\sigma=K|L\in\E}\mathcal P_{K,\sigma}[u,\phi](\chi_K^\eps)^2
  \le \mathcal D[u,\phi]. 
\end{align*}
Let $(N^k,P^k,Q^k,V^k)_{k=0,\ldots,N_T}$ be the discrete solution given by Theorem \ref{thm.ex}. We introduce the localized cross term
\begin{align} \label{5.Ctot}
  \mathcal C_{\rm tot}^{\eps,k} :=\sum_{\sigma=K|L\in\E}
  \big(\mathcal C_{K,\sigma}[N^k,-V^k]
  + \mathcal C_{K,\sigma}[P^k,V^k] 
  + \mathcal C_{K,\sigma}[Q^k,V^k]\big)(\chi_K^\eps)^2.
\end{align}


\subsection{A localized lower coercivity bound} 

For a discrete positive density $u$, we define the localized Fisher information $\mathcal I^\eps[u]$ and the total localized Fisher information $\mathcal I_{\rm tot}^{\eps,k}$ by
\begin{align}\label{5.Itot}
  \mathcal I^\eps[u] := \sum_{\sigma=K|L\in\E}\tau_\sigma
  |\D_\sigma\sqrt{u}|^2 (\chi_K^\eps)^2, \quad 
  \mathcal I_{\rm tot}^{\eps,k} := \mathcal I^\eps[N^k]
  + \mathcal I^\eps[P^k] + \mathcal I^\eps[Q^k].
\end{align}
The following lemma is a consequence of Proposition \ref{prop.first}. 

\begin{lemma}\label{lem.first}
It holds for $k=1,\ldots,N_T$ that
\begin{align*}
  \mathcal I_{\rm tot}^{\eps,k} + \mathcal C_{\rm tot}^{\eps,k}
  \lesssim \mathcal D[N^k,-V^k] + \mathcal D[P^k,V^k]
  + \mathcal D[Q^k,V^k],
\end{align*}
recalling 
definition \eqref{3.D} of the dissipation term $\mathcal D$.
\end{lemma}

\begin{proof}
We apply for $(P^k,V^k)$ the edgewise estimate in Proposition \ref{prop.first} and sum over $\sigma\in\E$:
\begin{align*}
  \sum_{\sigma=K|L\in\E}\mathcal P_{K,\sigma}[P^k,V^k](\chi_K^\eps)^2
  \gtrsim \mathcal I^\eps[P^k] + \sum_{\sigma=K|L\in\E}
  \mathcal C_{K,\sigma}[P^k,V^k](\chi_K^\eps)^2.
\end{align*}
Analogous inequalities hold for $N$ and $Q$, and summing them directly yields the result.
\end{proof}


\subsection{Control of the cross terms}

The cross terms in Lemma \ref{lem.first} can be controlled from below by the discrete electric energy, up to a factor of order $\eps^{-2}$ coming from the cutoff. To control the cutoff transition region, we use a discrete enlargement of $\chi^\eps$. This procedure was not needed in the continuous case \cite[Sec.~2.3]{JJZ23}, where the chain rule $|\na(\chi^\eps)^2| = 2\chi^\eps|\na\chi^\eps| \le C\eps^{-1}(\chi^\eps)^2$ was available. On the discrete level, this rule cannot be used. 

There exists a family $\widetilde\chi^\eps = (\widetilde\chi^\eps_K)_{K\in\T}$ such that 
\begin{align}\label{5.enlarge1}
  (\chi_K^\eps)^2\le (\widetilde\chi_K^\eps)^2
  &\quad\mbox{for }K\in\T, \\
  |\D_{K,\sigma}(\chi^\eps)^2| 
  + |(\widetilde\chi_K^\eps)^2 - (\chi_K^\eps)^2| \le 
  C\eps^{-1}\widetilde\chi_K^\eps h_K
  &\quad\mbox{for }\sigma\in\E_K, \label{5.enlarge2}
\end{align}
where $h_K:=\max_{\sigma\in\E_K}\dist_\sigma$. A concrete construction is given in Lemma \ref{lem.enlarge} in Appendix \ref{sec.app}. 

\begin{lemma}\label{lem.second}
There exist constants $C>0$ and $C'>0$, independent of $\eps$ and the discretization parameter $\eta$, such that for every $k=1,\ldots,N_T$,
\begin{align*}
  \mathcal C_{\rm tot}^{\eps,k} \ge -\frac{C}{\eps^2}
  \sum_{\sigma\in\E}\tau_\sigma|\D_{K,\sigma}V^k|^2 - C',
\end{align*}
where $C'$ depends only on the doping profile $A(x)$  and satisfies $C'=0$ if $A\equiv 0$.
\end{lemma}

\begin{proof}
We abbreviate $R^k:=N^k-P^k-Q^k$. By definition \eqref{5.Ctot} of the cross term and discrete integration by parts, we have
\begin{align*}
  & \mathcal C_{\rm tot}^{\eps,k} = -\sum_{\sigma=K|L\in\E}\tau_\sigma
  \D_{K,\sigma}R^k\D_{K,\sigma}V^k(\chi^\eps_K)^2
  = I_4 + I_5, \quad\mbox{where} \\
  & I_4 = \sum_{K\in\T}\sum_{\sigma\in\E_K}\tau_\sigma
  \D_{K,\sigma}V^k R_K^k(\chi^\eps_K)^2, \quad
  I_5 = \sum_{\sigma=K|L\in\E}\tau_\sigma\D_{K,\sigma}V^k
  R_K^k\D_{K,\sigma}(\chi^\eps)^2.
\end{align*}
We rewrite $I_4$ using the enlarged cutoff, leading to $I_4=I_{41}+I_{42}$, where
\begin{align*}
  I_{41} &= \sum_{K\in\T}\sum_{\sigma\in\E_K}\tau_\sigma
  \D_{K,\sigma}V^k R_K^k(\widetilde\chi^\eps_K)^2, \\
  I_{42} &= \sum_{K\in\T}\sum_{\sigma\in\E_K}\tau_\sigma
  \D_{K,\sigma}V^k R_K^k
  \big((\chi^\eps_K)^2 - (\widetilde\chi^\eps_K)^2\big).
\end{align*}
We infer from the discrete Poisson equation \eqref{3.V} and Young's inequality together with $A\in L^2(\Omega)$ that
\begin{align*}
  I_{41} = \sum_{K\in\T}\m(K)(R_K^k+A_K)R_K^k
  (\widetilde\chi^\eps_K)^2
  \ge \frac12\sum_{K\in\T}\m(K)|R_K^k|^2(\widetilde\chi^\eps_K)^2 - C.
\end{align*}
It follows from \eqref{5.enlarge2} and again Young's inequality that
\begin{align*}
  |I_{42}| &\le \frac{C}{\eps}\sum_{K\in\T}\sum_{\sigma\in\E_K}
  \tau_\sigma|\D_{K,\sigma}V^k||R_K^k|\widetilde\chi^\eps_K h_K \\
  &\le \frac14\sum_{K\in\T}\m(K)|R_K^k|^2(\widetilde\chi^\eps_K)^2
  + \frac{C}{\eps^2}\sum_{\sigma=K|L\in\E}
  \frac{\tau_\sigma^2 h_K^2}{\m(K)}|\D_{K,\sigma}V^k|^2.
\end{align*}
Lemma \ref{lem.meshregul} in Appendix \ref{sec.app} shows that $\tau_\sigma h_K^2/\m(K)$ is uniformly bounded such that
\begin{align*}
  I_4 = I_{41} + I_{42}
  \ge \frac14\sum_{K\in\T}\m(K)|R_K^k|^2(\widetilde\chi^\eps_K)^2 
  - \frac{C}{\eps^2}\sum_{\sigma=K|L\in\E}
  \tau_\sigma|\D_{K,\sigma}V^k|^2 - C. 
\end{align*}
Finally, we estimate $I_5$. We infer again from \eqref{5.enlarge2} that
\begin{align*}
  |I_5| &\le \frac{C}{\eps}\sum_{\sigma=K|L\in\E}\tau_\sigma
  |\D_{K,\sigma}V^k||R_K^k|\widetilde\chi^\eps_K h_K \\
  &\le \frac14\sum_{K\in\T}\m(K)|R_K^k|^2
  (\widetilde\chi^\eps_K)^2 + \frac{C}{\eps^2}
  \sum_{\sigma=K|L\in\E}\frac{\tau_\sigma^2 h_K^2}{\m(K)}
  |\D_{K,\sigma}V^k|^2. 
\end{align*}
Arguing similarly as before, and adding the bounds for $I_4$ and $I_5$ yields the result.  
\end{proof}


\subsection{Localized gradient estimates}

We combine the interior Fisher information estimate obtained in the previous subsection and the global bounds of Lemma \ref{lem.est1}. For a function $u=(u_K)\in\R^{\#\mathcal T}$, we define the localized product $u\chi^\eps := (u_K\chi_K^\eps)_{K\in\T}$. 

\begin{lemma}\label{lem.grad}
There exists a constant $C>0$, independent of $\eps$ and $\eta$, such that 
\begin{align}
  \sum_{k=1}^{N_T}\Delta t\sum_{\sigma=K|L\in\E}\m(\Delta_\sigma)
  \big(|\na_\sigma^{\Delta x}\sqrt{N^k}|^2 
  + |\na_\sigma^{\Delta x}\sqrt{P^k}|^2 
  + |\na_\sigma^{\Delta x}\sqrt{Q^k}|^2\big)(\chi^\eps_K)^2
  &\le \frac{C}{\eps^2}, \label{5.grad1} \\
  \sum_{k=1}^{N_T}\Delta t\sum_{\sigma=K|L\in\E}\m(\Delta_\sigma)
  \big(|\na_\sigma^{\Delta x}(N^k\chi^\eps)|^2 
  + |\na_\sigma^{\Delta x}(P^k\chi^\eps)|^2 
  + |\na_\sigma^{\Delta x}(Q^k\chi^\eps)|^2\big)
  &\le \frac{C}{\eps}. \label{5.grad2} 
\end{align}
\end{lemma}

\begin{proof}
We first prove \eqref{5.grad1}. A combination of Lemmas \ref{lem.first} and \ref{lem.second} leads to
\begin{align*}
  \mathcal I_{\rm tot}^{\eps,k} 
  \le C\big(\mathcal D[N^k,-V^k] + \mathcal D[P^k,V^k]
  + \mathcal D[Q^k,V^k]\big) + \frac{C}{\eps^2}\sum_{\sigma=K|L\in\E}
  \tau_\sigma|\D_{K,\sigma}V^k|^2 + C.
\end{align*}
We sum over $k=1,\ldots,N_T$ and use Lemma \ref{lem.est1} to obtain
$
  \sum_{k=1}^{N_T}\Delta t\mathcal I_{\rm tot}^{\eps,k} 
  \le C\eps^{-2},
$
which proves \eqref{5.grad1}. Next, the discrete product rule $\D_{K,\sigma}(N^k\chi^\eps) = \chi_K^\eps\D_{K,\sigma}N^k + N_{K,\sigma}\D_{K,\sigma}\chi^\eps$ implies that 
\begin{align*}
  & \sum_{k=1}^{N_T}\Delta t\sum_{\sigma=K|L\in\E}\m(\sigma)
  |\D_{K,\sigma}(N^k\chi^\eps)|\le I_6+I_7, \quad \mbox{where} \\
  & I_6 = \sum_{k=1}^{N_T}\Delta t\sum_{\sigma=K|L\in\E}\m(\sigma)
  \chi_K^\eps|\D_{K,\sigma}N^k|, \quad
  I_7 = \sum_{k=1}^{N_T}\Delta t\sum_{\sigma=K|L\in\E}\m(\sigma)
  N_{K,\sigma}^k|\D_{K,\sigma}\chi^\eps|.
\end{align*}
We first estimate $I_6$. We infer from
\begin{align*}
  \D_{K,\sigma}N^k = \big((N_K^k)^{1/2}+(N_{K,\sigma}^k)^{1/2}\big)
  \D_{K,\sigma}\sqrt{N^k} = 2(\overline{N}_\sigma)^{1/2}
  \D_{K,\sigma}\sqrt{N^k}
\end{align*}
and the Cauchy--Schwarz inequality that
\begin{align*}
  I_6 \le 2\bigg(\sum_{k=1}^{N_T}\Delta t\sum_{\sigma=K|L\in\E}
  \tau_\sigma|\D_{K,\sigma}\sqrt{N^k}|^2(\chi^\eps_K)^2\bigg)^{1/2}
  \bigg(\sum_{k=1}^{N_T}\Delta t\sum_{\sigma=K|L\in\E}\m(\sigma)
  \dist_\sigma(\overline{N}_\sigma^k)^2\bigg)^{1/2}.
\end{align*}
The first factor is bounded by $C/\eps$ thanks to \eqref{5.grad1}. For the second factor, we use $(\overline{N}_\sigma)^2 \le 2(N_K^k + N_{K,\sigma}^k)$, identity \eqref{2.mdelta}, and $\m(\Delta_\sigma)\le C\m(K)$ from Lemma \ref{lem.meshregul} in Appendix \ref{sec.app}:
\begin{align*}
  \sum_{\sigma\in\E}\m(\sigma)\dist_\sigma(\overline{N}_\sigma^k)^2
  = d\sum_{\sigma\in\E}\m(\Delta_\sigma)(\overline{N}_\sigma^k)^2
  \le C\sum_{K\in\T}\m(K)N_K^k.
\end{align*}
We conclude that $I_6\le C/\eps$. Next, we deduce from the cutoff bound \eqref{5.chieps}, the mesh regularity, and mass conservation that
\begin{align*}
  I_7 \le \frac{C}{\eps}\sum_{k=1}^{N_T}\Delta t\sum_{\sigma=K|L\in\E}
  \m(\sigma)\dist_\sigma N_{K,\sigma}^k \le \frac{C}{\eps}.
\end{align*}
Combining the estimates for $I_6$ and $I_7$ and proceeding analogously for $P$ and $Q$, we have shown \eqref{5.grad2}.
\end{proof}


\subsection{Localized discrete time derivative estimate}

We prove that the discrete time derivative $\pa_t^{\Delta t} u^k := (u^k-u^{k-1})/\Delta t$ is uniformly bounded.

\begin{lemma}
There exists $C>0$, independent of $\eps$ and $\eta$, such that
\begin{align}\label{5.time}
  \sum_{k=1}^{N_T}\Delta t\big(
  \|\pa_t^{\Delta t}(N^k\chi^\eps)\|_{-1,1,\T}^2
  + \|\pa_t^{\Delta t}(P^k\chi^\eps)\|_{-1,1,\T}^2
  + \|\pa_t^{\Delta t}(Q^k\chi^\eps)\|_{-1,1,\T}^2\big)
  \le \frac{C}{\eps^2}. 
\end{align}
\end{lemma}

\begin{proof}
Let $\zeta=(\zeta_K)\in\R^{\#\mathcal T}$ with $\|\zeta\|_{1,\infty,\T}\le 1$. We multiply equation \eqref{3.NPQ} for $N$ by $\chi_K^\eps\zeta_K$, sum over $K\in\T$, and apply discrete integration by parts:
\begin{align*}
  \sum_{K\in\T}\m(K)\pa_t^{\Delta t}(N_K^k\chi_K^\eps)\zeta_K
  = -\sum_{\sigma=K|L\in\E}\F_{K,\sigma}[N^k,-V^k]
  \D_{K,\sigma}(\chi^\eps\zeta).
\end{align*}
We use the discrete product rule, the bound $\|\zeta\|_{1,\infty,\T}\le 1$, and the cutoff estimate \eqref{5.chieps}:
\begin{align*}
  |\D_{K,\sigma}(\chi^\eps\zeta)|
  \le \chi_K^\eps|\D_{K,\sigma}\zeta| 
  + |\zeta_{K,\sigma}||\D_{K,\sigma}\chi^\eps|
  \le C(1+\eps^{-1})\dist_\sigma.
\end{align*}
This shows that
\begin{align*}
  \bigg|\sum_{K\in\T}\m(K)\pa_t^{\Delta t}
  (N_K^k\chi_K^\eps)\zeta_K\bigg|
  \le \frac{C}{\eps}\sum_{\sigma=K|L\in\E}\dist_\sigma
  |\F_{K,\sigma}[N^k,-V^k]|.
\end{align*}
We take the supremum over all admissible $\zeta$ and apply the Cauchy--Schwarz inequality:
\begin{align} 
  \|\pa_t^{\Delta t}(N^k\chi^\eps)\|_{-1,1,\T}
  &\le \frac{C}{\eps}\sum_{\sigma=K|L\in\E}\dist_\sigma
  |\F_{K,\sigma}[N^k,-V^k]| \nonumber \\
  &\le \frac{C}{\eps}\bigg(\sum_{\sigma=K|L\in\E}
  \frac{1}{\tau_\sigma\widetilde{N}_\sigma^k}
  |\F_{K,\sigma}[N^k,-V^k]|^2\bigg)^{1/2}
  \bigg(\sum_{\sigma\in\E}\tau_\sigma\widetilde{N}_\sigma^k
  \dist_\sigma^2\bigg)^{1/2} \label{5.aux} \\
  &\le \frac{C}{\eps}\mathcal D_{K,\sigma}[N^k,-V^k]^{1/2}
  \bigg(\sum_{\sigma\in\E}\tau_\sigma\widetilde{N}_\sigma^k
  \dist_\sigma^2\bigg)^{1/2}, \nonumber
\end{align}
where the last step follows from formula \eqref{2.P} and the definition of the dissipation term. Because of $\tau_\sigma = \m(\sigma)/\dist_\sigma$, we have $\tau_\sigma\dist_\sigma^2 = \m(\sigma)\dist_\sigma$, and the identity $d\m(\Delta_\sigma) = \m(\sigma)\dist_\sigma$ as well as the mesh regularity $\m(\Delta_\sigma)\le\xi_1^{-1}\m(K)$ from Lemma \ref{lem.meshregul} in Appendix \ref{sec.app} imply that
\begin{align*}
  \sum_{\sigma\in\E}\tau_\sigma\widetilde{N}_\sigma^k\dist_\sigma^2
  \le \sum_{\sigma\in\E}\m(\sigma)\dist_\sigma\widetilde{N}_\sigma^k
  = d\sum_{\sigma\in\E}\m(\Delta_\sigma)\widetilde{N}_\sigma^k
  \le C(\xi_1)\sum_{K\in\T}\m(K)N_K^k \le C,
\end{align*}
where the last step follows from mass conservation (see Lemma \ref{lem.est1}). We conclude from \eqref{5.aux} and Lemma \ref{lem.est1} after summation over $k=1,\ldots,N_T$ that 
\begin{align*}
  \sum_{k=1}^{N_T}\Delta t
  \|\pa_t^{\Delta t}(N^k\chi^\eps)\|_{-1,1,\T}^2
  \le \frac{C}{\eps}\sum_{k=1}^{N_T}\Delta t
  \mathcal D_{K,\sigma}[N^k,-V^k] \le \frac{C}{\eps}.
\end{align*}
The same arguments apply to $P^k$ and $Q^k$. This proves the lemma.
\end{proof}


\section{Compactness and convergence}\label{sec.comp}

Let $\mathcal M_m=(\T_m,\E_m,\mathcal P_m;\Delta t_m,\Delta x_m)$ be a sequence of admissible meshes with mesh size $\eta_m\to 0$ as $m\to\infty$, and let $(N^k,P^k,Q^k,V^k)_{k=0,\ldots,N_T^m}$ be the corresponding finite-volume solution to \eqref{3.NPQ}--\eqref{3.bc}. To simplify the notation, we set $\pa_t^m:=\pa_t^{\Delta t_m}$, $\na^m:=\na^{\Delta x_m}$, $\pi_m:=\pi_{\eta_m}$, and $\pi_m^*:=\pi_{\eta_m}^*$. We denote by $(N_m,P_m,Q_m,V_m):=(\pi_m N,\pi_m P,\pi_m Q,\pi_m V)$ the piecewise constant reconstructions and correspondingly $\na_m V:=\pi_m^*(\na^m V)$, $\na_m\sqrt{N}:=\pi_m^*(\na^m\sqrt{N})$ with analogous notations for $P$ and $Q$. 

\subsection{Compactness of the approximate solutions}

The first lemma shows the compactness of the approximate densities.

\begin{lemma}\label{lem.convgrad}
There exist nonnegative functions $N^*$, $P^*$, $Q^*\in L^1(\Omega_T)$ such that, up to a subsequence, as $m\to\infty$, 
\begin{align*}
  N_m\to N^*,\quad P_m\to P^*, \quad Q_m\to Q^*\quad
  \mbox{strongly in }L^1(\Omega_T). 
\end{align*} 
Furthermore, for every open subset $\omega\Subset\Omega\cup\Gamma_N$,
\begin{align*}
  \na_m\sqrt{N}\rightharpoonup\na\sqrt{N^*}, \quad
  \na_m\sqrt{P}\rightharpoonup\na\sqrt{P^*}, \quad
  \na_m\sqrt{Q}\rightharpoonup\na\sqrt{Q^*}
\end{align*}
weakly in $L^2((0,T)\times\omega;\R^d)$.
\end{lemma}

\begin{proof}
We infer from \eqref{5.grad2} and \eqref{5.time} that
\begin{align*}
  \sum_{k=1}^{N_T^m}\Delta t_m\big(\|N^k\chi^\eps\|_{1,1,\T_m}
  + \|\pa_t^m(N^k\chi^\eps)\|_{-1,1,\T_m}\big) \le C(\eps). 
\end{align*}
Hence, by the discrete Aubin--Lions lemma \cite[Theorem 3.4]{GaLa12} and the arguments in \cite[Sec.~5.2]{JuZu23}, the sequence $(\pi_m(N\chi^\eps))_{m\in\N}$ is relatively compact in $L^1(\Omega_T)$. Thus, there exists a subsequence which is not relabeled such that, as $m\to\infty$,
\begin{align*}
  \pi_m(N\chi^\eps) \to N^\eps \quad\mbox{strongly in }L^1(\Omega_T). 
\end{align*}
It follows from $\chi^\eps\equiv 1$ in $\Omega_D^{2,\eps}$ that
\begin{align*}
  N_m = \pi_m N\to N^\eps \quad\mbox{strongly in }
  L^1((0,T)\times\Omega_D^{2\eps}).
\end{align*}
We apply as in \cite[Sec.~2.4]{JJZ23} a Cantor diagonal argument to obtain the existence of a function $N^*\in L^1(\Omega_T)$ such that
\begin{align*}
  N_m\to N^*\quad\mbox{strongly in }L^1_{\rm loc}(\Omega_T).
\end{align*}
 We claim that $(N_m)$ converges globally in $L^1(\Omega_T)$. Indeed, we know from Lemma \ref{lem.est1} that $(N_m)$ is uniformly bounded in $L\log L(\Omega_T)$ and, by the De la Vall\'ee--Poussin theorem, uniformly integrable in $\Omega_T$. We infer from Vitali's theorem that 
\begin{align*}
  N_m\to N\quad\mbox{strongly in }L^1(\Omega_T). 
\end{align*}

Next, let $\omega\Subset\Omega\cup\Gamma_N$. Choose $\eps>0$ such that $\omega\subset\Omega_D^{2\eps}$. The weak convergence $\pi_m^*\na^m \sqrt{N}\rightharpoonup\na \sqrt{N^*}$ in $L^2(0,T;L^2(\omega))$ (possibly for a subsequence) is a consequence of the uniform bound \eqref{5.grad1} and the arguments of \cite[Lemma 4.4]{CLP03}. 
\end{proof}

We state the compactness of the electric potential.

\begin{lemma}\label{lem.V}
There exists a function $V^*\in L^\infty(0,T;H^1(\Omega))$ such that, up to a subsequence,
\begin{align*}
  V_m\rightharpoonup V^*\quad\mbox{weakly* in }
  L^\infty(0,T;H^1(\Omega))\quad\mbox{as }m\to\infty.
\end{align*}
In particular, $\na_mV\rightharpoonup \na V^*$ weakly* in $L^\infty(0,T;L^2(\Omega;\R^d))$. 
\end{lemma}

\begin{proof}
By Lemma \ref{lem.est1}, the sequence $(V_m)$ is bounded in $L^\infty(0,T;H^1(\Omega))$. Hence, there exists a subsequence (not relabeled) such that, as $m\to\infty$,
\begin{align*}
  V_m\rightharpoonup V^*\ \mbox{weakly* in }
  L^\infty(0,T;L^2(\Omega)), \quad
  \na_mV\rightharpoonup \Phi\ \mbox{weakly* in }
  L^\infty(0,T;L^2(\Omega;\R^d)).
\end{align*} 
The arguments of \cite[Lemma 4.4]{CLP03} show that $\Phi=\na V^*$. 
\end{proof}

Finally, we prove the convergence of the edge densities.

\begin{lemma}\label{lem.P12}
There exists a subsequence (not relabeled) such that
\begin{align*}
  \pi_m^*\big(\overline{P}^{1/2}\big)\to (P^*)^{1/2}, \quad
  \pi_m^*\bigg(\frac{\widehat{P}}{\overline{P}^{1/2}}\bigg)
  \to (P^*)^{1/2}\quad\mbox{strongly in }L^2(\Omega_T),
\end{align*}
and similar convergences hold for $\overline{N}$ and $\overline{Q}$. 
\end{lemma}

\begin{proof}
We split the difference of $\pi_m^*(\overline{P}^{1/2})$ and $P_m^{1/2}$ into two parts:
\begin{align*}
  & \int_0^T\int_\Omega\big|\pi_m^*(\overline{P}^{1/2}) 
  - P_m^{1/2}\big|^2 dxdt
  = J_{1,m}^\eps + J_{2,m}^\eps, \quad\mbox{where} \\
  & J_{1,m}^\eps := \int_0^T\int_{\Omega_D^\eps}
  \big|\pi_m^*(\overline{P}^{1/2}) - P_m^{1/2}\big|^2 dxdt, \\
  & J_{2,m}^\eps := \int_0^T\int_{\Omega\backslash\Omega_D^\eps}
  \big|\pi_m^*(\overline{P}^{1/2}) - P_m^{1/2}\big|^2 dxdt.
\end{align*}
For the first term, let $\sigma=K|L\in\E_K$ and $(t,x)\in(t_{k-1},t_k]\times\Delta_\sigma$. Then
\begin{align}\label{5.aux2}
  \big|\pi_m^*(\overline{P}^{1/2})(t,x) - P_m^{1/2}(t,x)\big|
  = \big|(\overline{P}_\sigma)^{1/2} - (P_K^k)^{1/2}\big|
  \le \frac12\big|(P_{L})^{1/2} - (P_K^k)^{1/2}\big|.
\end{align}
Hence, by \eqref{5.grad1},
\begin{align*}
  J_{1,m}^\eps \le C\eta_m^2\int_0^T\int_{\Omega_D^\eps}
  |\na_m \sqrt{P}|^2 dxdt \le \frac{C}{\eps^2}\eta_m^2.
\end{align*}
For the boundary-layer term, we infer from \eqref{5.aux2} that
\begin{align*}
  \big|\pi_m^*(\overline{P}^{1/2})(t,x) - P_m^{1/2}(t,x)\big|
  \le \frac12\big((P_{L}^k)^{1/2} + (P_K^k)^{1/2}\big)
  = 2(\overline{P}_\sigma^k)^{1/2} = 2\pi_m^*(\overline{P}^{1/2}).
\end{align*}
Moreover, using the convexity of $s\mapsto s^2\log s^2$, Jensen's inequality, and Lemma \ref{lem.est1},
\begin{align*}
  \int_0^T\int_\Omega&\pi_m^*(\overline{P})\log\pi_m^*(\overline{P})
  dxdt = \sum_{k=1}^{N_T^m}\Delta t_m\sum_{\sigma=K|L\in\E_m} 
  \m(\Delta_\sigma)\overline{P}_\sigma^k\log \overline{P}_\sigma^k \\
  &\le C\sum_{k=1}^{N_T^m}\Delta t_m\sum_{\sigma=K|L\in\E_m}
  \m(\Delta_\sigma)\big(P_L^k\log P_L^k + P_K^k\log P_K^k\big)
  \le C.,
\end{align*} 
Hence, the sequence $(\pi_m^*(\overline{P}))_m$ is uniformly integrable in $\Omega_T$, and thus, $\sup_{m\in\N}J_{2,m}^\eps\to 0$ as $\eps\to 0$. We choose $\eps=\eta_m^{1/2}$ to find that
\begin{align*}
  \pi_m^*(\overline{P}^{1/2}) - P_m^{1/2}\to 0 
  \quad\mbox{strongly in }L^2(\Omega_T). 
\end{align*}
We deduce from $P_m^{1/2}\to(P^*)^{1/2}$ strongly in $L^2(\Omega_T)$ (see Lemma \ref{lem.convgrad}) that
\begin{align*}
  \pi_m^*(\overline{P}^{1/2}) \to (P^*)^{1/2}\quad
  \mbox{strongly in }L^2(\Omega_T).
\end{align*}

The second convergence in the lemma follows from the same arguments, taking into account that
\begin{align*}
  \bigg|\frac{\widehat{P}_\sigma}{\overline{P}_\sigma^{1/2}}
  - \overline{P}_\sigma^{1/2}\bigg|
  = \frac{\widehat{P}_\sigma^{1/2} + \overline{P}_\sigma^{1/2}}{
  \overline{P}_\sigma^{1/2}}
  \big|\widehat{P}_\sigma^{1/2} - \overline{P}_\sigma^{1/2}\big|
  \le C|P_{K,\sigma}^{1/2} - P_K^{1/2}|
\end{align*}
for every edge $\sigma\in\E_K$.
\end{proof}


\subsection{Global convergence of the square-root gradients}

We now identify the limits of the discrete fluxes and derive the global convergence of the discrete gradients of the square root of the densities. We first note that the Bernoulli factor converges to one as $m\to\infty$.

\begin{lemma}\label{lem.B}
It holds that $\pi_m^*(B(|\D_{K,\sigma}V|))\to 1$ strongly in $L^2(\Omega_T)$ as $m\to\infty$. 
\end{lemma} 

\begin{proof}
The inequalities $0\le 1-B(s)\le s/2$ for $s\ge 0$ imply that
\begin{align*}
  \|\pi_m^*(B(|\D_{K,\sigma}V|))-1\|_{L^2(\Omega_T)}^2
  \le C\int_0^T\int_\Omega\eta_m^2|\na_m V|^2 dxdt \le C\eta_m^2
  \to 0,
\end{align*}
using the uniform boundedness of $(\na_m V)$ in $L^2(\Omega_T)$. 
\end{proof}

Next, we identify the limit of the flux combinations. Recall definition \eqref{4.widehatN} of $\widehat{N}_m$, $\widehat{P}_m$, and $\widehat{Q}_m$. We set $\D_\sigma V_m := (\D_{K,\sigma}V_K^k)_{K\in\T_m,k=1,\ldots,N_T^m}$.

\begin{lemma}\label{lem.flux2}
It holds that
\begin{align*}
  \pi_m^*\bigg(2B(|\D_{\sigma}V_m|)\na^m\sqrt{N}
  - \frac{\widehat{N}}{\overline{N}^{1/2}}\na^m V\bigg)
  &\rightharpoonup\Phi_N, \\
  \pi_m^*\bigg(2B|\D_{\sigma}V_m|)\na^m\sqrt{P}
  + \frac{\widehat{P}}{\overline{P}^{1/2}}\na^m V\bigg)
  &\rightharpoonup\Phi_P, \\
  \pi_m^*\bigg(2B(|\D_{\sigma}V_m|)\na^m\sqrt{Q}
  + \frac{\widehat{Q}}{\overline{Q}^{1/2}}\na^m V\bigg)
  &\rightharpoonup\Phi_Q
\end{align*}
weakly in $L^2(\Omega_T)$, where the limits are given by
\begin{align*}
  \Phi_N&= 2\na\sqrt{N^*} - \sqrt{N^*}\na V^*, \\
  \Phi_P&= 2\na\sqrt{P^*} + \sqrt{P^*}\na V^*, \\
  \Phi_Q&= 2\na\sqrt{Q^*} + \sqrt{Q^*}\na V^*.
\end{align*}
\end{lemma}

\begin{proof}
We know from Lemma \ref{lem.flux} that the sequence
\begin{align*}
  \pi_m^*\bigg(2B(|\D_{\sigma}V_m|)\na^m\sqrt{P}
  + \frac{\widehat{P}}{\overline{P}^{1/2}}\na^m V\bigg)
\end{align*}
is bounded in $L^2(\Omega_T)$, so it converges (up to a subsequence) weakly in $L^2(\Omega_T)$ to some function $\Phi_P$. To identify this limit, let $\varphi\in C_0^\infty(\Omega_T;\R^d)$. We write
\begin{align*}
  & \int_0^T\int_\Omega\pi_m^*\bigg(2B(|\D_{\sigma}V_m|)\na^m\sqrt{P}
  + \frac{\widehat{P}}{\overline{P}^{1/2}}\na^m V\bigg)\cdot\varphi
  dxdt = J_{3,m} + J_{4,m}, \quad\mbox{where} \\
  & J_{3,m} = 2\int_0^T\int_\Omega \pi_m^*(B(|\D_{\sigma}V_m|))
  \na_m\sqrt{P}\cdot\varphi dxdt, \\
  & J_{4,m} = \int_0^T\int_\Omega\pi_m^*\bigg(\frac{\widehat{P}}{
  \overline{P}^{1/2}}\na^m V\bigg)\cdot\varphi dxdt.
\end{align*}
Lemma \ref{lem.B} implies that $\pi_m^*(B(|\D_\sigma V_m|))\to 1$ strongly in $L^2(\Omega_T)$, while Lemma \ref{lem.convgrad} gives $\na_m\sqrt{P}\rightharpoonup\na\sqrt{P^*}$ weakly in $L^2((0,T)\times\omega;\R^d)$ for every $\omega\Subset\Omega\cup\Gamma_N$. In particular, since $\varphi$ has compact support,
\begin{align*}
  \na_m\sqrt{P}\cdot\varphi\rightharpoonup\na\sqrt{P^*}\cdot\varphi
  \quad\mbox{weakly in }L^2(\Omega_T).
\end{align*}
Thus, the product of $\pi_m^*(B(|\D_\sigma V_m|))$ and $\na_m\sqrt{P}$ converges weakly in $L^1(\Omega_T)$:
\begin{align*}
  J_{3,m}\to 2\int_0^T\int_\Omega\na\sqrt{P^*}\cdot\varphi dxdt.
\end{align*}

For the term $J_{4,m}$, we observe that by Lemma \ref{lem.P12},
$\pi_m^*(\widehat{P}/\overline{P}^{1/2})\to \sqrt{P^*}$ strongly in $L^2(\Omega_T)$, while Lemma \ref{lem.V} yields $\na_m V\rightharpoonup\na V^*$ weakly* in $L^\infty(0,T;L^2(\Omega;\R^d))$. Hence,
\begin{align*}
  J_{4,m}\to \int_0^T\int_\Omega \sqrt{P^*}\na V^*\cdot\varphi dxdt.
\end{align*}
This shows the result for $P^*$. The cases $N^*$ and $Q^*$ are treated in the same way.
\end{proof}

We are able to extend the local result of Lemma \ref{lem.convgrad} globally in $\Omega$. 

\begin{lemma}
It holds that
\begin{align*}
  \na_m\sqrt{N}\rightharpoonup \na\sqrt{N^*}, \quad 
  \na_m\sqrt{P}\rightharpoonup \na\sqrt{P^*}, \quad
  \na_m\sqrt{Q}\rightharpoonup \na\sqrt{Q^*}
\end{align*}
weakly in $L^1(\Omega_T;\R^d)$. 
\end{lemma}

\begin{proof}
Lemma \ref{lem.flux} provides an $L^1(\Omega_T)$ bound for $\na_m\sqrt{N}$, $\na\sqrt{P}$, and $\na_m\sqrt{Q}$. This is not sufficient to conclude weak convergence, and we need to exploit the dissipation structure in more detail. More precisely, we decompose $\na_m\sqrt{P} = J_{5,m}+J_{6,m}+J_{7,m}$, where
\begin{align*}
  J_{5,m} &= -\frac12\pi_m^*\bigg(\frac{\widehat{P}}{\overline{P}^{1/2}}\bigg)
  \na^m V, \quad
  J_{6,m} = \big(1-\pi_m^*(B(|\D_\sigma V_m|))\big)
  \na_m\sqrt{P}, \\
  J_{7,m} &= \frac12\pi_m^*\bigg(2B(|\D_{K,\sigma} V_m|)
  \na_m\sqrt{P} + \frac{\widehat{P}}{\overline{P}^{1/2}}
  \na^m V\bigg).
\end{align*}
It follows from Lemma \ref{lem.flux2} that
\begin{align*}
  J_{7,m} \rightharpoonup \na\sqrt{P^*} + \frac12\sqrt{P^*}\na V^*
  \quad \mbox{weakly in }L^2(\Omega_T;\R^d),
\end{align*}
and we infer from Lemmas \ref{lem.V} and \ref{lem.P12} that
\begin{align*}
  J_{5,m}\rightharpoonup -\frac12\sqrt{P^*}\na V^*
  \quad \mbox{weakly in }L^1(\Omega_T;\R^d).
\end{align*}
We claim that $J_{6,m}\rightharpoonup 0$ weakly in $L^1(\Omega_T;\R^d)$,
which shows the result.

To show the claim, let $\varphi\in L^\infty(\Omega_T;\R^d)$. Because of $0\le 1-B(s)\le s/2$ for $s\ge 0$, 
\begin{align*}
  \bigg|\int_0^T\int_\Omega J_{6,m}\cdot\varphi dxdt\bigg|
  \le \frac12\int_0^T\int_\Omega\pi_m^*(|\D_{K,\sigma}V_m|)
  |\na_m\sqrt{P}||\varphi|dxdt.
\end{align*}
Let $\eps>0$ and split the integral over $\Omega_T= ((0,T)\times\Omega_D^\eps)\cup((0,T)\times(\Omega\backslash\Omega_D^\eps))$. On the first set, we use $|\D_{K,\sigma} V_m|\le C\eta_m|\na_m V|$ and the local $L^2$ bound for $\na_m\sqrt{P}$ from \eqref{5.grad1}. This yields
\begin{align*}
  \bigg|\int_0^T\int_{\Omega_D^\eps}J_{6,m}\cdot\varphi dxdt\bigg|
  \le\frac{C}{\eps}\eta_m\|\varphi\|_{L^\infty(\Omega_T)}. 
\end{align*} 
On the boundary layer, we estimate $|\D_\sigma V_m||\na_m\sqrt{P}|\le 2|\na_m V|\pi_m^*(\overline{P}^{1/2})$ and use the $L^2(\Omega_T)$ bound for $\na_mV$ from Lemma \ref{lem.est1} as well as the uniform integrability of $\pi_m^*(\overline{P}^{1/2})$ from Lemma \ref{lem.P12} to conclude that
\begin{align*}
  \sup_{m\in\N}\bigg|\int_0^T\int_{\Omega\backslash\Omega_D^\eps}
  J_{6,m}\cdot\varphi dxdt\bigg|\to 0 \quad\mbox{as }\eps\to 0.
\end{align*} 
Thus, choosing $\eps=\sqrt{\eta_m}$, we obtain $J_{6,m}\rightharpoonup 0$ weakly in $L^1(\Omega_T;R^d)$. 
\end{proof}


\subsection{Passage to the limit}

We now pass to the limit in the finite-volume scheme. First, we provide an auxiliary lemma on the convergence of the discrete gradients. A weak* consistency of the reconstructed discrete gradient is established in \cite[Lemma 3.7]{ScSe22}. Here, we need an integrated first-order consistency estimate at the level of a diamond cell.

\begin{lemma}[Consistency of the discrete gradient]\label{lem.cons}
Let $\zeta\in C_0^\infty(\Omega_T)$ and define $\zeta_K^k:=\zeta(t_k,x_K)$ for $K\in\T_m$ and $k=0,\ldots,N_T^m$. Then, there exists $C>0$ independent of $m$ such that for $k=1,\ldots,N_T^m$ and $\sigma\in\E_m$,
\begin{align*}
  \bigg|\Delta t_m\m(\Delta_\sigma)\na_\sigma^m\zeta^k
  - d\int_{t_{k-1}}^{t_k}\int_{\Delta_\sigma}\na\zeta(t,x)dxdt\bigg|
  \le C\eta_m\Delta t_m\m(\Delta_\sigma)
  \|\zeta\|_{W^{2,\infty}(\Omega_T)}.
\end{align*}
\end{lemma}

\begin{proof}
Let $(t,x)\in(t_{k-1},t_k]\times\Delta_\sigma$ for $\sigma=K|L$. A Taylor expansion gives $\zeta_K^k-\zeta_{L}^k = \na\zeta(t,x)\cdot(x_K-x_{L}) + O(\eta_m\dist_\sigma)$ and hence,
\begin{align*}
  \Delta t_m\m(\Delta_\sigma)(\zeta_K^k-\zeta_{L}^k)
  = (x_K-x_L)\cdot\int_{t_{k-1}}^{t_k}\int_{\Delta_\sigma}\na\zeta dxdt
  + O(\eta_m\dist_\sigma\Delta t_m\m(\Delta_\sigma)).
\end{align*}
This shows that
\begin{align*}
  \bigg|&d\Delta t_m\m(\Delta_m)\na_\sigma^m\zeta^k
  - d\int_{t_{k-1}}^{t_k}\int_{\Delta_\sigma}\na\zeta(t,x)dxdt\bigg| \\
  &= d\bigg|\Delta t_m\m(\Delta_\sigma)
  \frac{\zeta_K^k-\zeta_L^k}{\dist_\sigma}\frac{x_K-x_L}{|x_K-x_L|}
  - \int_{t_{k-1}}^{t_k}\int_{\Delta_\sigma}\na\zeta(t,x)dxdt\bigg|  
  = O(\eta_m\Delta t_m\m(\Delta_\sigma)),
\end{align*}
ending the proof.
\end{proof}

\begin{lemma}[Identifying the continuity equations]
The limit function $(N^*,P^*,Q^*,V^*)$ satisfies the continuity equations in \eqref{1.eq}. 
\end{lemma}

\begin{proof}
Let $\varphi$ be an admissible test function in the weak formulation of the equation for $N^*$. We set $\varphi_K^k:=\varphi(t_k,x_K)$ for $K\in\T_m$ and $k=0,\ldots,N_T^m$. Multiply the equation for $N$ in \eqref{3.NPQ} by $\Delta t_m\phi_K^{k-1}$, sum over $K\in\T_m$ and $k=1,\ldots,N_T^m$, and perform discrete integration by parts to find that $J_{8,m}+J_{9,m}=0$, where
\begin{align*}
  J_{8,m} &= \sum_{k=1}^{N_T^m}\sum_{K\in\T_m}\m(K)
  (N_K^k-N_K^{k-1})\varphi_K^{k-1}, \\
  J_{9,m} &= \sum_{k=1}^{N_T^m}\Delta t_m\sum_{\sigma=K|L\in\E_m}
  \tau_\sigma\D_{K,\sigma}\varphi^{k-1}
  \big(B(|\D_{K,\sigma} V^k|)\D_{K,\sigma}N^k - \widehat{N}_\sigma^k
  \D_{K,\sigma}V^k\big).
\end{align*}

We reformulate the first term by using summation by parts:
\begin{align*}
  J_{8,m} &= -\sum_{k=1}^{N_T^m}\sum_{K\in\T_m}\m(K)N_K^k
  (\varphi_K^k-\varphi_K^{k-1}) 
  - \sum_{K\in\T_m}\m(K)N_K^0\varphi_K^0 \\
  &= -\int_0^T\int_\Omega N_m\pa_t\varphi dxdt
  - \int_\Omega N_m^0\varphi(0,\cdot)dx + R_1, \quad\mbox{where} \\
  R_1 &:= \sum_{k=1}^{N_T^m}\sum_{K\in\T_m}\m(K)\int_{t_{k-1}}^{t_k}
  N_K^k\bigg(\pa_t\varphi(t,x) 
  - \frac{\varphi_K^k-\varphi_K^{k-1}}{\Delta t_m}\bigg)dt.
\end{align*}
Since $\varphi$ is smooth and $(N_m)$ is bounded in $L^1(\Omega_T)$, we find that
\begin{align*}
  |R_1|\le C\Delta t_m\|\pa_t^2\varphi\|_{L^\infty(\Omega_T)}
  \|N_m\|_{L^1(\Omega_T)}\to 0\quad\mbox{as }m\to\infty.
\end{align*}

We reformulate $J_{9,m}$ as
\begin{align*}
  J_{9,m} &= \int_0^T\int_\Omega L_m\cdot\na\varphi dxdt
  + R_2, \quad\mbox{where} \\
  L_m &:= \pi_m^*(\overline{N}^{1/2})\pi_m^*\bigg(
  2B(|\D_\sigma V_m|)\na^m\sqrt{N} 
  - \frac{\widehat{N}_m}{\overline{N}_m^{1/2}}\na^m V\bigg), \\
  R_2 &:= \frac{1}{d}\sum_{k=1}^{N_T^m}\sum_{\sigma\in\E_\sigma}
  \bigg(\Delta t_m\m(\Delta_\sigma)\na_\sigma^m\varphi^{k-1}
  - d\int_{t_{k-1}}^{t_k}\int_{\Delta_\sigma}\na\varphi dxdt\bigg),
  \cdot L_m^k,
\end{align*}
where $L_m^k$ denotes the value of $L_m$ on $(t_{k-1},t_k]\times\Delta_\sigma$. We infer from Lemma \ref{lem.cons} that 
\begin{align*}
  |R_2|\le C\eta_m\|\varphi\|_{W^{2,\infty}(\Omega_T)}
  \|L_m\|_{L^1(\Omega_T)}. 
\end{align*}
The $L^1(\Omega_T)$ norm of $L_m$ is bounded since, by Lemmas \ref{lem.flux} and \ref{lem.P12}, 
\begin{align*}
  \|L_m\|_{L^1(\Omega_T)} 
  &\le \big\|\pi_m^*\big(\overline{N}^{1/2}\big)\big\|_{L^2(\Omega_T)}
  \bigg\|\pi_m^*\bigg(2B(|\D_{K,\sigma}V_m|)\na^m \sqrt{N}
  - \frac{\widehat{N}_m}{\overline{N}_m^{1/2}}\na^m V
  \bigg)\bigg\|_{L^2(\Omega_T)} \\
  &\le C.
\end{align*}
We conclude that $R_2\to 0$ as $m\to\infty$. 

Combining the previous estimates, we arrive at
\begin{align}\label{5.aux3}
  -\int_0^T\int_\Omega N_m\pa_t\varphi dxdt 
  - \int_\Omega N_m^0\varphi(0,\cdot)dx
  + \int_0^T\int_\Omega L_m\cdot\na\varphi dxdt = o(1).
\end{align}
We know from Lemma \ref{lem.convgrad} that $N_m\to N^*$ strongly in $L^1(\Omega_T)$ and clearly $N_m^0\to N^{\rm in}$ strongly in $L^1(\Omega)$. Moreover, by Lemmas \ref{lem.P12} and \ref{lem.flux2}, 
\begin{align*}
  L_m\rightharpoonup \sqrt{N^*}\big(2\na\sqrt{N^*}
  - \sqrt{N^*}\na V^*\big) = \na N^* - N^*\na V^*
  \quad\mbox{weakly in }L^1(\Omega_T).
\end{align*}
Thus, the limit $m\to\infty$ implies that
\begin{align*}
  -\int_0^T\int_\Omega N^*\pa_t\varphi dxdt 
  - \int_\Omega N^{\rm in}\varphi(0,\cdot)dx
  + \int_0^T\int_\Omega(\na N^* - N^*\na V^*)\cdot\na\varphi dxdt = 0.
\end{align*}
This is the weak formulation of the first continuity equation in \eqref{1.eq}. The continuity equations for $P^*$ and $Q^*$ are derived in the same way.
\end{proof}

\begin{lemma}[Identification of the Poisson equation]
The limit function $(N^*,P^*,Q^*,$ $V^*)$ satisfies the Poisosn equation in \eqref{1.eq}.
\end{lemma}

\begin{proof}
Let $\psi$ be a test function in the weak formulation of the Poisson equation and set $\psi_K^k:=\psi(t_k,x_K)$ for $K\in\T_m$, $k=0,\ldots,N_T^m$. We multiply the discrete Poisson equation \eqref{3.V} by $\Delta t_m\psi_K$ and sum over $K\in\T_m$, $k=1,\ldots,N_T^m$. Then $J_{10,m}+J_{11,m}=0$, where
\begin{align*}
  J_{10,m} 
  &:= \lambda^2\sum_{k=1}^{N_T^m}\Delta t_m\sum_{\sigma=K|L\in\E_m}
  \tau_\sigma\D_{K,\sigma}\psi^k\D_{K,\sigma}V^k, \\
  J_{11,m} &:= \sum_{k=1}^{N_T^m}\Delta t_m\sum_{K\in\T_m}\m(K)
  \psi_K(N_K^k-P_K^k-Q_K^k+A_K).
\end{align*}
The second term becomes
\begin{align*}
  & J_{11,m} = \int_0^T\int_\Omega(N_m-P_m-Q_m+\pi_m A)\psi dxdt + R_3,
  \quad\mbox{where} \\
  & |R_3|\le C\eta_m\|\psi\|_{W^{1,\infty}(\Omega_T)}
  \big(\|N_m\|_{L^1(\Omega_T)} + \|P_m\|_{L^1(\Omega_T)}
  + \|Q_m\|_{L^1(\Omega_T)} + \|A\|_{L^1(\Omega)}\big),
\end{align*}
and consequently $R_3\to 0$ as $m\to\infty$. For the first term, we write
\begin{align*}
  & J_{10,m} = \lambda^2\int_0^T\int_\Omega\na_m V\cdot\na\psi dxdt
  + R_4, \quad\mbox{where} \\
  & R_4 := \frac{\lambda^2}{d}\sum_{k=1}^{N_T^m}\sum_{\sigma\in\E_m}
  \bigg(\Delta t_m\m(\Delta_\sigma)\na_\sigma^m\psi^k
  - d\int_{t_k}^{t_{k+1}}\int_{\Delta_\sigma}\na\psi dxdt\bigg)
  \cdot\na_\sigma^m V^k.
\end{align*}
It follows from Lemma \ref{lem.cons} and the uniform $L^2(\Omega_T)$ bound for $\na_m V$ that
\begin{align*}
  |R_4|\le C\eta_m\|\psi\|_{W^{2,\infty}(\Omega_T)}
  \|\na_m V\|_{L^1(\Omega_T)}\to 0 \quad\mbox{as }m\to\infty.
\end{align*}

The previous arguments show that 
\begin{align*}
  \lambda^2\int_0^T\int_\Omega\na_m V\cdot\na\psi dxdt
  + \int_0^T\int_\Omega(N_m-P_m-Q_m+\pi_m A)\psi dxdt = o(1).
\end{align*}
The convergence results in Lemmas \ref{lem.convgrad} and \ref{lem.V} as well as the convergence $\pi_m A\to A$ strongly in $L^2(\Omega)$ yield in the limit $m\to\infty$ that
\begin{align*}
  \lambda^2\int_0^T\int_\Omega\na V^*\cdot\na\psi dxdt
  + \int_0^T\int_\Omega(N^*-P^*-Q^*+A)\psi dxdt = 0,
\end{align*}
which is the weak formulation of the Poisson equation. 
\end{proof}

It remains to identify the Dirichlet boundary conditions. 

\begin{lemma}[Dirichlet boundary conditions]
It holds that
\begin{align*}
  & (N^*)^{1/2} - (N^D)^{1/2}, \quad
  (P^*)^{1/2} - (P^D)^{1/2}\in L^1(0,T;W_D^{1,1}(\Omega)), \\
  & V^*-V^D\in L^\infty(0,T;H_D^1(\Omega)). 
\end{align*}
\end{lemma}

\begin{proof}
The sequence $W_m:=V_m-\pi_m V^D$ vanishes on the Dirichlet boundary and is uniformly bounded in $L^\infty(0,T;H^1(\Omega))$ by Lemma \ref{lem.est1}. Extending $W_m$ by zero across $\Gamma_D$ and using the standard finite-volume extension argument, we infer that, up to a subsequence,
\begin{align*}
  W_m\rightharpoonup W\quad\mbox{weakly* in }L^\infty(0,T;H^1(\Omega'))
\end{align*} 
for a slightly larger domain $\Omega'\supset\Omega$ with $W=0$ outside of $\Omega$ near $\Gamma_D$. We conclude from $W_m\rightharpoonup V^*-V^D$ on $\pa\Omega$ that $V^*-V^D\in L^\infty(0,T;H_D^1(\Omega))$. 

Next, we consider the boundary values for the densities. The function $Z_m=N_m^{1/2}-(\pi_m N^D)^{1/2}$ vanishes on $\Gamma_D$. We know from Lemma \ref{lem.convgrad} that $N_m^{1/2}$ converges strongly in $L^2(\Omega_T)$, and Lemma \ref{lem.P12} gives the weak convergence of their gradients in $L^1(\Omega_T)$. Extending $Z_m$ by zero across $\Gamma_D$ and applying the same trace argument as before, we infer that $(N^*)^{1/2}-(N^D)^{1/2}\in L^1(0,T;W^{1,1}_D(\Omega))$. Similarly, it follows that $(P^*)^{1/2}-(P^D)^{1/2}\in L^1(0,T;W^{1,1}_D(\Omega))$.
\end{proof}


\section{Numerical experiments}\label{sec.num}

We illustrate the behavior of the numerical scheme by performing some experiments. Our objective is not to develop a quantitatively accurate physical model, but rather to illustrate the qualitative behavior of the oxygen-vacancy density.

We consider the two-dimensional square $\Omega=(0,1)^2$ and choose a uniform grid of $2\times 80\times 80 = 12,800$ triangles. Taking the physical parameters as in \cite[Sec.~6]{JJZ23}, the scaled Debye length becomes $\lambda^2 = 2.86\cdot 10^{-4}$, corresponding to a silicon device with the size $50\,\mbox{nm}\times 50\,\mbox{nm}$. The other physical parameters are:
\begin{itemize}
\item Doping profile: The bulk semiconductor is undoped with $A=0$, while we take $A=10$ close to the Ohmic contacts; see Figure \ref{fig.dev}. 
\item Vacancy profile: Two highly doped vacancy seeds are placed in the middle of the bulk with $Q=500$, while $Q=50$ close to the contacts and $Q=0$ elsewhere.
\item Applied bias: The electric potential equals $U_0=3$ at the bottom electrode and $U_0=2$ at the top electrode, corresponding to a rather small 26\,mV potential difference.
\end{itemize}
We observe that the inclusion of oxygen seeds is common in experimental studies \cite{KCD26}. The Dirichlet data and the built-in potential are taken as in \cite[Sec.~6]{JJZ23}. The initial state is defined by assuming local charge neutrality ($N - P - Q + A = 0$) and thermal equilibrium ($N P = 1$), which avoids nonphysical transients at $t = 0$. 

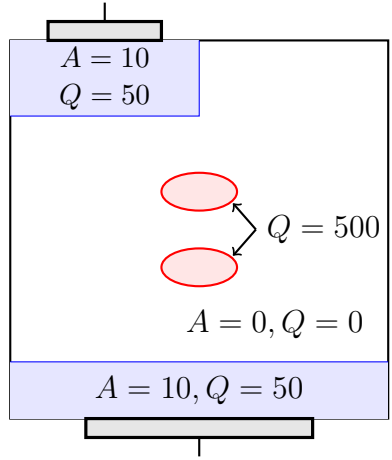
\begin{figure}[ht]
    \centering
    \begin{tikzpicture}[scale=5]
        \draw[thick] (0,0) rectangle (1,1);
        \fill[blue!10] (0,0.8) rectangle (0.5,1);
        \fill[blue!10] (0,0) rectangle (1,0.15);
        \fill[red!10] (0.5, 0.6) ellipse (0.1 and 0.05);
        \fill[red!10] (0.5, 0.4) ellipse (0.1 and 0.05);
        \draw[blue] (0,0.8) -- (0.5,0.8) -- (0.5,1);
        \draw[blue] (0,0.15) -- (1,0.15);
        \draw[thick, red] (0.5, 0.6) ellipse (0.1 and 0.05);
        \draw[thick, red] (0.5, 0.4) ellipse (0.1 and 0.05);
        \node[align=center] at (0.25, 0.9) {\small $A=10$\\ \small $Q=50$};
        \node at (0.5, 0.075) {$A=10, Q=50$};
        \node at (0.7, 0.25) {$A=0, Q=0$};
        \node[right] at (0.65, 0.5) {$Q=500$};
        \draw[->, thick, shorten >=2pt] (0.65, 0.5) -- (0.58, 0.58);
        \draw[->, thick, shorten >=2pt] (0.65, 0.5) -- (0.58, 0.42);
        \draw[very thick, fill=gray!20] (0.1, 1) rectangle (0.4, 1.05);
        \draw[thick] (0.25, 1.05) -- (0.25, 1.1);
        \draw[very thick, fill=gray!20] (0.2, 0) rectangle (0.8, -0.05);
        \draw[thick] (0.5, -0.05) -- (0.5, -0.1);
    \end{tikzpicture}
\caption{Schematic presentation of the doping profile $A(x)$ and initial oxygen vacancy density $Q(x)$, featuring two localized high-concentration vacancy seeds in the bulk.}
\label{fig.dev}
\end{figure}
Because of the small Debye length, the fully implicit nonlinear discrete system \eqref{3.NPQ}--\eqref{3.bc}, which needs to be solved in every time step, is highly stiff such that a monolithic implicit solver would be computationally very expensive. Therefore, at each time step, we solve equations \eqref{3.NPQ}--\eqref{3.bc} by a Gummel iteration. More precisely, we stabilize the discrete Poisson equation by solving at the $\ell$th iteration
\begin{align*}
  -\lambda^2&\sum_{\sigma\in\E_K}\tau_\sigma\D_{K,\sigma}V^{(\ell+1)}
  + \m(K)\big(N_K^{(\ell)}+P_K^{(\ell)}+Q_K^{(\ell)}\big)
  V_K^{(\ell+1)} \\
  &= - \m(K)\big(N_K^{(\ell)}-P_K^{(\ell)}-Q_K^{(\ell)}+A_K\big)
  + \m(K)\big(N_K^{(\ell)}+P_K^{(\ell)}+Q_K^{(\ell)}\big)V_K^{(\ell)}.
\end{align*} 

This adjustment ensures that the system matrix is diagonally dominant, successfully suppressing spurious oscillations in the electrostatic field. The $(\ell+1)$th iteration of the electron density is computed as follows:
\begin{align*}
  \frac{\m(K)}{\Delta t}(N^{(\ell+1)}-N^k) 
  - \sum_{\sigma\in\E_K}\F_{K,\sigma}(N^{(\ell+1)},-V^{(\ell+1)}) = 0,
\end{align*}
with an analogous approach for $P^{(\ell+1)}$ and $Q^{(\ell+1)}$. The iteration repeats until the relative difference is smaller than $10^{-3}$ or the 200th iteration is reached.

The simulation starts with the very small initial time step size $\Delta t = 10^{-7}$ to ensure that the Gummel fixed-point solver converges during the initial highly nonlinear transient phase, where the gradients are very steep. After the first 5 time steps, $\Delta t$ dynamically increases by 10\% per time step to improve computational efficiency, capped at a maximum of $\Delta t_{\rm max} = 5\cdot 10^{-5}$.

First, we verify our implementation by computing the numerical convergence rate in one space dimension. We use the configuration shown in Figure \ref{fig.one} and the simulation time $T=0.1$. The numerical solution $Q^{[N]}$ is computed on uniform grids with $N\in\{50,100,200,400,800\}$ spatial cells. The discrete error is evaluated against a reference solution, calculated with $N=2500$ cells. The reference solution $Q_{\rm ref}$ is linearly interpolated onto the cell centers of each coarser mesh. The $L^1$-error is defined by ${\rm err}_N = \sum_{K\in\T_h}\m(K)|Q_K^{[N]}-Q_{\rm ref}(x_K)|$. Figure \ref{fig.rate} shows that the spatial convergence rate is approximately one as expected. 

\begin{figure}[htbp]
    \centering
    \begin{tikzpicture}[x=12cm, y=1.2cm] 
        \draw[thick, fill=white] (0,0) rectangle (1,1);
        \fill[blue!10] (0,0) rectangle (0.15,1);
        \fill[blue!10] (0.85,0) rectangle (1,1);
        \fill[red!10] (0.35,0) rectangle (0.45,1);
        \fill[red!10] (0.55,0) rectangle (0.65,1);
        \draw[thick, blue] (0.15,0) -- (0.15,1);
        \draw[thick, blue] (0.85,0) -- (0.85,1);
        \draw[thick, red] (0.35,0) -- (0.35,1);
        \draw[thick, red] (0.45,0) -- (0.45,1);
        \draw[thick, red] (0.55,0) -- (0.55,1);
        \draw[thick, red] (0.65,0) -- (0.65,1);
        \draw[thick] (0,0) rectangle (1,1);
        \foreach \x in {0, 0.15, 0.35, 0.45, 0.55, 0.65, 0.85, 1} {
            \draw (\x, 0) -- (\x, -0.1) node[below, font=\scriptsize] {\x};
        }
        \node[below] at (0.5,-0.6) {Spatial coordinate $x$};
        \node[align=center] at (0.075, 0.5) {\scriptsize $A=10$\\ \scriptsize $Q=50$};
        \node[align=center] at (0.925, 0.5) {\scriptsize $A=10$\\ \scriptsize $Q=50$};
        \node[align=center] at (0.4, 0.5) {\scriptsize $Q=500$};
        \node[align=center] at (0.6, 0.5) {\scriptsize $Q=500$};
        \node[align=center] at (0.25, 0.5) {\scriptsize $A=0$\\ \scriptsize $Q=0$};
        \node[align=center] at (0.75, 0.5) {\scriptsize $A=0$\\ \scriptsize $Q=0$};
        \node[align=center] at (0.5, 0.5) {\scriptsize $A=0$\\ \scriptsize $Q=0$};
    \end{tikzpicture}
    \caption{Schematic presentation of the one-dimensional domain used for the spatial convergence analysis.}
    \label{fig.one}
\end{figure}
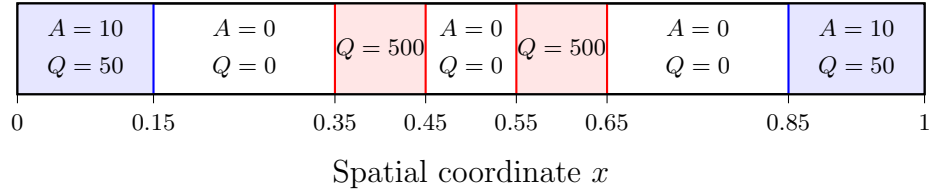

\begin{figure}[ht]
\includegraphics[width=80mm]{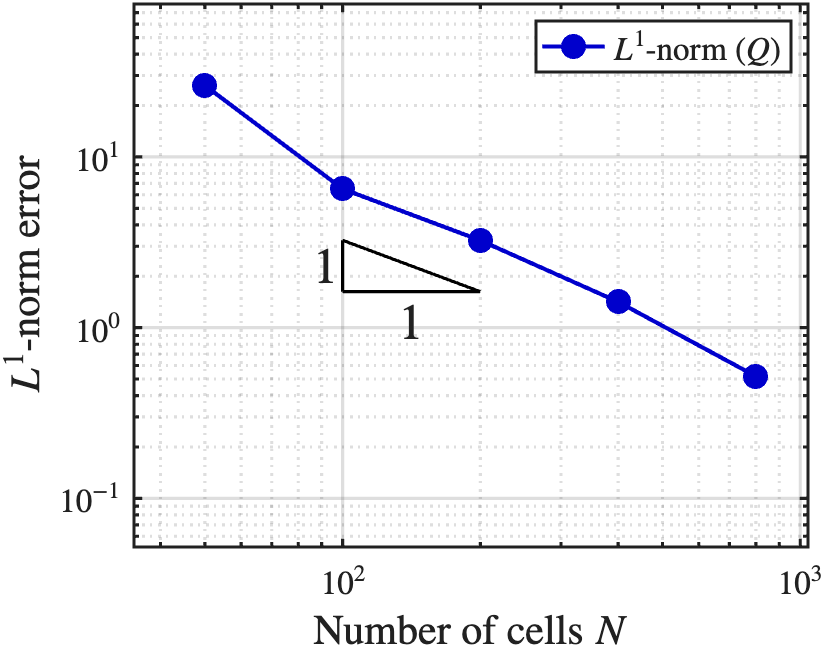}
\caption{Discrete $L^1$-error of the oxygen vacancy density $Q$ as a function of the number of spatial cells $n$ in a log-log plot.}
\label{fig.rate}
\end{figure}

Next, we compute the relative energy associated with the oxygen vacancy density $Q^k$ at discrete time $t_k=k\Delta t$,
\begin{align*}
  \mathcal{E}(t_k) = \sum_{K\in\T_h}\m(K)\bigg(Q_K^k
  \log\frac{Q_K^k}{Q^\infty} - Q_K^k + Q^\infty\bigg),
\end{align*}
where $Q^\infty$ is the equilibrium density. The discrete evolution of $\mathcal{E}(t_k)$ is shown in Figure \ref{fig.ent}. We observe that the numerical scheme preserves the monotonicity of the internal energy of the oxygen vacancy density. 

\begin{figure}[ht]
\includegraphics[width=120mm]{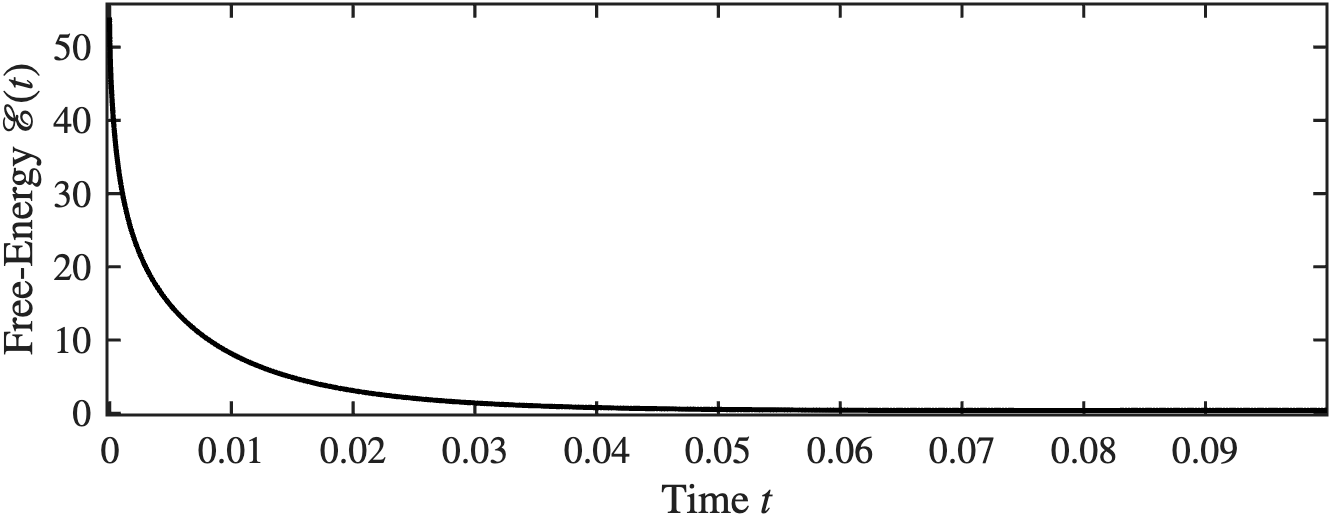}
\caption{Evolution of the relative free energy associated with the oxygen vacancy density.}
\label{fig.ent}
\end{figure}

Figure \ref{fig.oxy} illustrates the spatial distribution of the oxygen-vacancy density over time. The vacancies diffuse from the seed regions and form a high-density pathway connecting the top and bottom electrodes, which may be interpreted as a conductive filament analogous to those observed experimentally. Over time, however, this structure broadens due to diffusion, and the vacancy density converges towards a nonhomogeneous equilibrium state. The results indicate that a stable filament cannot be sustained within the present diffusion-only framework and that additional physical mechanisms would be required to stabilize the filamentary structure. More realistic simulations reported in the literature are typically based on atomistic approaches \cite{AKS26} or kinetic Monte--Carlo methods \cite{BrSh26}. Since the development of such models lies beyond the scope of the present work, we leave this extension for future research.

\begin{figure}[ht]
\includegraphics[width=150mm]{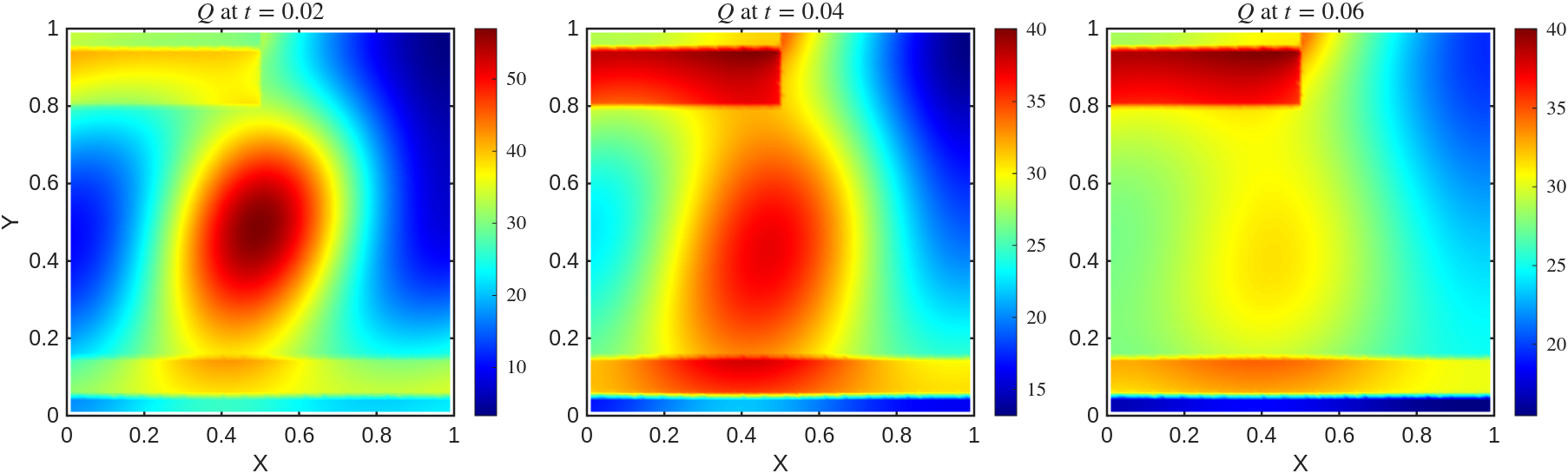}
\caption{Snapshots of the oxygen vacancy density $Q$ at times $t=0.02$ (left), $t=0.04$ (middle), and $t=0.06$ (right).}
\label{fig.oxy}
\end{figure}


\begin{appendix}
\section{Auxiliary results}\label{sec.app}

We first prove some consequences of the mesh regularity \eqref{2.meshreg}.

\begin{lemma}\label{lem.meshregul}
Let $(\T,\E,\mathcal X)$ be an admissible mesh satisfying \eqref{2.meshreg} and set $h_K:=\max_{\sigma\in\E_K}\dist_\sigma$. Then
\begin{align*}
  \m(\Delta_\sigma)\le \xi_1^{-1}\m(K), \quad
  \frac{\tau_\sigma h_K^2}{\m(K)}\le \frac{d}{\xi_1\xi_2^2}
  \quad\mbox{for }K\in\T,\ \sigma\in\E_K.
\end{align*}
\end{lemma}

\begin{proof}
Denoting by $\Delta_{K,\sigma}$ the half-diamond associated with $K$ and $\sigma$, we have
\begin{align}\label{a.aux}
  \m(\Delta_{K,\sigma}) = \frac{1}{d}\m(\sigma)d(x_K,\sigma), \quad
  \sum_{\sigma\in\E_K}\m(\Delta_{K,\sigma}) = \m(K).
\end{align}
Hence, because of \eqref{2.mdelta} and \eqref{2.meshreg},
\begin{align*}
  d\m(\Delta_\sigma) = \m(\sigma)\dist_\sigma
  \le\xi_1^{-1}\m(\sigma)d(x_K,\sigma) 
  = d\xi_1^{-1}\m(\Delta_{K,\sigma})
  \le d\xi_1^{-1}\m(K),
\end{align*}
which proves the first statement. Next, we use \eqref{a.aux} and then \eqref{2.meshreg} to estimate
\begin{align*}
  \frac{\tau_\sigma h_K^2}{\m(K)} 
  = \frac{\m(\sigma)}{\dist_\sigma}\frac{h_K^2}{\m(K)} 
  \le \frac{\m(\sigma)}{\dist_\sigma}\frac{h_K^2}{\m(\Delta_{K,\sigma})}
  = \frac{d}{\dist_\sigma}\frac{h_K^2}{d(x_K,\sigma)}
  \le \frac{d}{\xi_1}\frac{h_K^2}{\dist_\sigma^2}
  \le \frac{d}{\xi_1\xi_2^2},
\end{align*}
which shows the second statement.
\end{proof}

Next, we construct the discrete enlarged cutoff function used in the proof of Lemma \ref{lem.first}. 

\begin{lemma}\label{lem.enlarge}
Let $\chi^\eps\in C_0^\infty(\Omega_D^\eps)$ be as defined in Section \ref{sec.chi}. Set $\chi_K^\eps=\chi^\eps(x_K)$ and $\chi_\sigma^\eps=\chi^\eps(x_\sigma)$ for $K\in\T$, $\sigma\in\E_{\rm ext}$, where $x_\sigma=\m(\sigma)\int_\sigma xdx$. We construct the enlarged cutoff $\widetilde\chi^\eps$ by
\begin{align*}
  (\widetilde\xi_K^\eps)^2 := \max\Big\{(\chi_K^\eps)^2,
  \max_{\sigma\in\E_K}(\chi^\eps_{K,\sigma})^2\Big\},
\end{align*}
which satisfies estimates \eqref{5.enlarge1}--\eqref{5.enlarge2}. 
\end{lemma}

\begin{proof}
Estimate \eqref{5.enlarge1} follows from the definition of $\widetilde\chi^\eps$. To prove \eqref{5.enlarge2}, we deduce from
\begin{align*}
  |\D_{K,\sigma}(\chi^\eps)^2|
  = |(\chi^\eps_{K,\sigma})^2 - (\chi^\eps_K)^2|
  = (\chi^\eps_{K,\sigma}+\chi^\eps_K)|\D_{K,\sigma}\chi^\eps|,
  \quad \chi_{K,\sigma}^\eps+\chi_K^\eps\le 2\widetilde\chi^\eps_K,
\end{align*}
and inequality \eqref{5.chieps} that
\begin{align*}
  |\D_{K,\sigma}(\chi^\eps)^2| \le 2\widetilde\chi_K^\eps  
  |\D_{K,\sigma}\chi^\eps| \le C\eps^{-1}\widetilde\chi_K^\eps h_K,
\end{align*}
using $\dist_\sigma\le h_K$ in the last step. Next, we estimate
\begin{align*}
  |(\widetilde\chi_K^\eps)^2 - (\chi_K^\eps)^2|
  &= (\widetilde\chi_K^\eps 
  + \chi^\eps_K)|\widetilde\chi^\eps_K-\chi^\eps_K|
  \le 2\widetilde\chi_K^\eps|\widetilde\chi^\eps_K-\chi^\eps_K| \\
  &\le 2\widetilde\chi_K^\eps\max_{\sigma\in\E_K}|\D_{K,\sigma}\chi^\eps|
  \le C\eps^{-1}\widetilde\chi_K^\eps h_K,
\end{align*}
finishing the proof.
\end{proof}

\end{appendix}


\end{document}